\documentclass[11pt,reqno]{amsart}

\usepackage{amsmath, amsfonts, amsthm, amssymb}

\newtheorem{Theorem}{Theorem}[section]
\newtheorem{Lemma}{Lemma}[section]
\newtheorem{Proposition}{Proposition}[section]

\theoremstyle{definition}
\newtheorem{Definition}{Definition}[section]

\theoremstyle{remark}
\newtheorem{Remark}{Remark}[section]

\numberwithin{equation}{section}

\renewcommand{\r}{\rho}

\renewcommand{\u}{{\bf u}}

\newcommand{\R}{{\mathbb R}}
\newcommand{\Dv}{{\nabla\cdot\,}}

\newcommand{\x}{{\bf x}}

\def\f{\frac}
\renewcommand{\O}{\Omega}

\def\D{\Delta }

\def\hf1{^\f{1}{1-\xi^2}}

\def\be{\begin{equation}}
\def\en{\end{equation}}
\def\bs{\begin{split}}
\def\es{\end{split}}

\renewcommand{\d}{{\bf d}}
\newcommand{\F}{{\bf d}}

\renewcommand{\a}{\alpha}

\author{Xiaoli Li and\ Dehua Wang}
\address{Institute of Applied Physics and Computational Mathematics,
                           Beijing 100088, China.}
\email{xllithu@gmail.com}

\address{Department of Mathematics, University of Pittsburgh,
                           Pittsburgh, PA 15260, USA.}
\email{dwang@math.pitt.edu}

\title[Global Solution to the Flow of Liquid Crystals]
{Global strong Solution to the  Density-Dependent Incompressible Flow of
Liquid  Crystals}



\date{January 2, 2012}%\today}

\begin{document}
\begin{abstract}
The initial-boundary value problem for the density-dependent incompressible flow of liquid crystals is studied
in a three-dimensional bounded smooth domain.
%The system consists of a coupled system of incompressible inhomogeneous Navier-Stokes equations and various
%kinematic transport equations for the molecular orientations.
For the initial density  away from vacuum, the existence and
uniqueness is established for both the local strong solution with
large initial data and the global strong solution with small data.
It is also proved that when the strong solution exists, a
weak solution with the same data must be equal to the unique strong
solution.
\end{abstract}

\maketitle

\section{Introduction}

\iffalse%%%%%
Liquid crystals are a state of matter that have properties between
those of a conventional liquid and those of a solid crystal that are
optically anisotropic, even when they are at rest. In this work, we
are interested in a Navier-Stokes type model for density-dependent
viscous incompressible fluids that takes into account the
crystallinity of the fluid molecules in 3-D case, that is, a nematic
liquid crystal model, which can be governed by the following
nonlinear hydrodynamical system
\fi%%%%

Liquid crystals are substances that exhibit a phase of matter that has properties between those of a conventional liquid and those of a solid crystal. A liquid crystal  may flow like a liquid, but its molecules may be oriented in a crystal-like way. There are many different types of liquid crystal phases, which can be distinguished based on their different optical properties.
One of the most common liquid crystal phases is the nematic, where the molecules have no positional order, but they have long-range orientational order.
The three-dimensional density-dependent incompressible flow of nematic liquid crystals can be governed by the following system of partial differential equations
(\cite{Chan,JLE,Leslie1,LL}):
\begin{subequations}\label{e2}
\begin{align}
&\partial_t\r+\nabla\cdot(\r\u)=0, \label{e21}\\
&\partial_t(\r\u)+\nabla\cdot(\r\u\otimes\u)+\nabla P
       =\mu\triangle\u-\lambda\nabla\cdot\left(\nabla \d\odot\nabla\d\right), \label{e22}\\
&\partial_t\d+\u\cdot\nabla\d=\gamma\left(\D\d+|\nabla\d|^2\d\right), \label{e23}\\
&\nabla\cdot\u=0, \label{e24}
\end{align}
\end{subequations}
where $\rho$ denotes the density, $\u\in\R^3$   the velocity, $\d\in \mathbb{S}^2$ (the unit sphere in $\R^3$)   the unit-vector field that represents the  macroscopic molecular orientations,
 $P\in\R$   the pressure (including both the hydrostatic part and the induced elastic part from the orientation field arising from
  the incompressibility $\nabla\cdot\u=0$);    they all depend on the spatial variable  $\x \in\R^3$ and the time variable $t>0$.
 %$\sigma=PI_{3\times3}+\mu\nabla\u$ is the fluid viscosity part of the stress tensor.
 The positive constants $\mu, \lambda, \gamma$ stand for viscosity, the competition between kinetic energy and potential energy,
 and microscopic elastic relaxation time or the Deborah number for the molecular orientation field, respectively.
 The term $\lambda\nabla\cdot\left(\nabla  \d\odot\nabla\d\right)$ in the stress tensor represents the
 anisotropic feature of the system.
  We set $\mu=\lambda=\gamma=1$ since their exact values do not play any role in our analysis.
The symbol $\nabla\d\odot\nabla\d$ denotes a matrix whose $(i,j)$-th
entry is $\partial_{x_i} \d\cdot\partial_{x_j}\d$ for $1\leq i,j\leq
3$, and it is easy to see that $\nabla\d\odot\nabla\d=(\nabla
\d)^\top\nabla\d,$ where $(\nabla \d)^\top$ denotes the transpose of
the $3\times 3$ matrix $\nabla \d$.
System \eqref{e2} is a simplified version, but still retains most of
the interesting mathematical properties (without destroying the
basic nonlinear structure) of the original Ericksen-Leslie model
(\cite{Eri, Eri2, HK, HKL, Leslie1, Lin2}) for the hydrodynamics of
nematic liquid crystals; see \cite{DLWW,DWW,LL,LW,SL} for more
discussions on the relations of the two models. Both the
Ericksen-Leslie system and the simplified one \eqref{e2} describe
the macroscopic continuum time evolution of liquid crystal materials
under the influence of both the velocity and the orientation of
crystals  which can be derived from the averaging/coarse
graining of the directions of rod-like liquid crystal molecules. In
particular, there is a force term in the $\u$-system \eqref{e22}
depending on $\d$; the left-hand side of the $\d$-system \eqref{e23}
stands for the kinematic transport by the flow field,  while the
right-hand side represents the internal relaxation due to the
elastic energy. In many situations, the flow velocity field does
disturb the alignment of   molecules, and in turn a change in the
alignment will induce velocity.

We consider the initial-boundary value problem of system \eqref{e2}
in a bounded smooth domain $\O\subset\R^3$ with the initial condition:
\begin{equation}\label{ic}
(\r,\u, \d)\mid_{t=0}=(\r_0,\u_0,\d_0), \quad  \x\in  {\Omega},
\end{equation}
  and the boundary condition:
\begin{equation}\label{bc}
(\u,\partial_{\bf \nu}\d)\mid_{\partial\Omega}=(0,0), % \quad t>0.
\end{equation}
where $\rho_0: \O\to\R^+$, $\u_0: \O\to\R^3$, $\d_0: \O\rightarrow \mathbb{S}^2$ are given with
compatibility,   and  ${\bf \nu}$ denotes
the outer unit-normal vector field on $\partial\O$.
The boundary condition implies  non-slip on the boundary
and no contribution to the surface forces  from the director field $\d$.
% for the orientation vector $\d$, homogeneous Neumann
%boundary condition (which represents the fact that there is no
%contribution to the surface forces, $\sigma\cdot\nu$, from the
%director field d) is considered here.
%
Roughly speaking,  \eqref{e2} is a coupling between the
incompressible inhomogeneous Navier-Stokes equations and the
transported flow of harmonic maps. In the homogeneous case $\r\equiv
1$, \eqref{e2} becomes the hydrodynamic flow system of
incompressible liquid crystals. In a series of papers \cite{LL}-\cite{LL4},  Lin and Liu
 addressed both the regularity and existence of global weak
solutions to the Leslie system of variable length, i.e.  when the Dirichlet energy
$$\frac{1}{2}\int_\O|\nabla\d|^2 d\x, \qquad \d: \O\rightarrow \mathbb{S}^{2},$$
 is replaced by the Ginzburg-Landau energy
$$\int_\O(\frac{1}{2}|\nabla\d|^2+\frac{(1-|\d|^2)^2}{4\varepsilon^2}) d\x \ (\varepsilon>0), \qquad \d: \O\rightarrow \R^3.$$
In particular, for any fixed $\varepsilon>0$, they \cite{LL} proved the global existence of
weak solutions with large initial data under the assumption that
$\u_0\in L^2(\O), \d_0\in H^1(\O)$ with
$\d_0\mid_{\partial\Omega}\in H^{\frac{3}{2}}(\partial\Omega)$ in
the two-dimensional and three-dimensional cases.  The existence and uniqueness of global classical
solution was also obtained if $\u_0\in H^1(\O), \d_0\in H^2(\O)$   when the fluid viscosity $\mu$ is large enough.
%However, the estimates and arguments there depend on $\varepsilon$, and it is a challenging problem to study the convergence as $\varepsilon\downarrow 0$.  The similar results were obtained also in \cite{SL} for a different but similar model.
The partial regularity   of the weak solution was investigated in \cite{LL2} (and also \cite{CGR,HKL,LW}), similar
to the classical theorem by Caffarelli, Kohn, and Nirenburgh
\cite{CKN} on the Navier-Stokes equations that asserts the one-dimensional
parabolic Hausdorff measure of the singular set of any suitable weak solution is zero.
% It is also an open problem to ask whether there exists a global weak solution to the incompressible hydrodynamic
%flow equation similar to the Leray-Hopf type solutions in the context of Navier-Stokes equation.
With the Ginzburg-Landau penalty function, the global strong and weak solutions to  the compressible flow of liquid crystals were obtained in \cite{LLa,LLH,LLQ,WY1}.
See also \cite{CS1,HW6,JT,LW,SL} for some related discussions.
For the incompressible version of system \eqref{e2} with constant density, Lin-Lin-Wang
\cite{LLW}  established the existence of global weak solutions that are smooth away from at most
finitely many singular times in any bounded smooth domain of $\R^2$,
and we \cite{LXWD} proved the global existence of strong solution  in a bounded smooth domain of $\R^3$.
For the compressible version of system \eqref{e2}, the one-dimensional classical solution was obtained  in \cite{DLWW,DWW,QH},
%the existence of a weak solution was obtained in \cite{DLWW} under the assumption that the initial density function $\r_0\in H^1$ has a positive lower bound (also in \cite{DWW} under the weaker assumption that $0\leq \r_0\in L^\chi$ for $\chi>1$), $\u_0\in L^2$ and $\d_0\in H^1$.
and the blowup criteria of strong solutions  were studied in \cite{HWW1, HWW2}.

In this paper,  we are interested in the existence and uniqueness of
global strong solution $(\r, \u, P, \d) $ of \eqref{e2} in
$W^{1,r}(\O)\times W^{2,q}(\O)^3\times
W^{1,q}(\O)\times W^{3,q}(\O)^3$ with $3<q\leq r\leq \infty$ while assuming in addition
that the initial density is bounded away from zero. By a
$\textit{strong solution}$, we mean a quadruplet $(\r, \u, P, \d)$
satisfying \eqref{e2} almost everywhere with the initial-boundary
conditions \eqref{ic}-\eqref{bc}. Our strategy is to consider the
following auxiliary problem:
\begin{equation*}
\begin{cases}
\partial_t\r+{\bf v}\cdot\nabla\r=0,\\
\r\partial_t\u-\D\u+\nabla
P=-\r{\bf v}\cdot\nabla{\bf v}-\nabla\cdot\big((\nabla{\bf f})^\top\nabla{\bf f}\big), \\
\partial_t\d-\D\d=-{\bf v}\cdot\nabla{\bf f}+|\nabla{\bf f}|^2{\bf f},\\
\nabla\cdot\u=0,
\end{cases}
\end{equation*}
for some given ${\bf v}\in\R^3$ and ${\bf f}\in\R^3$. One of the
motivations for  such a strategy is that the continuity
equation \eqref{e21} is the transport equation of $\r$, \eqref{e22}
is the evolutionary density-dependent incompressible Navier-Stokes
equation with the source term $-\lambda\nabla\cdot\left(\nabla
\d\odot\nabla\d\right)$, while \eqref{e23} is the parabolic system
in terms of $\d$, therefore we can use a
 result of the transport equation (cf. Proposition
\ref{P1}), the maximal regularities of the parabolic equations (cf.
Theorem \ref{T3}) and density-dependent Stokes equations (cf.
Theorem \ref{T4}). We first use an iteration method to establish the
local existence and uniqueness of strong solution with general
initial data. Then we prove the global existence by establishing
some global estimates under the condition that the initial data are
small in some sense. As system \eqref{e2} contains the Navier-Stokes
equations as a subsystem, one cannot expect  in general  any better
results than those for the Navier-Stokes equations. The uniqueness
of global weak solution is always an open problem. We shall prove
that when the strong solution exists, all the global weak solutions
must be equal to the unique strong solution, which is called the
weak-strong uniqueness. Similar results were obtained by Danchin
\cite{D} for the density-dependent incompressible viscous fluids in
a bounded domain of $\R^2$ with $C^{2+\varepsilon}$ boundary.
%In fact, when the orientation vector is a constant unit vector, \eqref{e2} reduces to the density-dependent incompressible Navier-Stokes equation.
We shall establish our results in the
spirit of \cite{D}, while developing new estimates for the crystal
orientation field. Due to the particular structure of the equations
for the velocity, especially the strongly nonlinear term
$(\nabla\d)^\top\triangle\d$ in the $\u$-system, it will be
necessary to obtain more regularities for the crystal orientation
field. By developing more novel  and subtle estimates, we will be able to finally establish the global existence of strong solution and weak-strong uniqueness for the initial-boundary value problem  \eqref{e2}-\eqref{bc} of the density-dependent incompressible flow of liquid crystals. The results of this paper generalize our early results in \cite{LXWD} for the incompressible case with constant density. The analysis in this paper is much more difficult and complicated than that in \cite{LXWD} due to the appearance of non-constant density.

The rest of the  paper is organized as follows. In Section 2, we
state our main results on local and global existence of strong
solution, as well as the weak-strong uniqueness. In Section 3, we
recall a standard result for the transport equation, the maximal
regularities for the non-homogeneous non-stationary Stokes operator
and the parabolic operator, and also some $L^\infty$ estimates in
the spatial variable. In Section 4, we give the proof of the local
existence. In Section 5, we prove the global existence. Finally in
Section 6, we show the weak-strong uniqueness.

\bigskip

%%%%%%%%%%%%%%%%
\section{Main Results}

In this section, we state our main results. If $k>0$ is an integer
and $p\ge 1$, we denote by $W^{k,p}$ the set of functions in
$L^p(\O)$ whose derivatives of up to order $k$ belong to $L^p(\O)$.
For $T>0$ and a function space $X$, denote by $L^p(0,T; X)$ the set
of Bochner measurable $X$-valued time dependent functions $f$ such
that $t\rightarrow \|f\|_{X}$ belongs to $L^p(0,T)$, and the
corresponding Lebesgue norm is denoted by $\|\cdot\|_{L^p_T(X)}$.
We will consider the solutions in the functional spaces defined below.

\begin{Definition}\label{df1}
For $T>0$ and $1<p, q,r<\infty$, we denote by $M^{p,q,r}_T$ the set
of quadruplets $(\r,\u,P,\d)$ such that
$$\u\in C([0,T]; D_{A_q}^{1-\f{1}{p},p})\cap L^p(0,T;
W^{2,q}(\O)\cap W^{1,q}_0(\O)), \ \ \partial_t \u \in
L^p(0,T;L^q(\O)),\; \nabla\cdot\u=0;$$
$$\d\in C([0,T]; B_{q,p}^{3(1-\f{1}{p})})\cap L^p(0,T;W^{3,q}(\O)), \ \ \partial_t \d\in L^p(0,T;L^q(\O));$$
$$\r\in C([0,T]; W^{1,r}(\Omega));\ \  P\in L^p(0,T;
W^{1,q}(\O))\  \ \textrm{and}\   \int_\O P\ d\x=0.$$ If $r=\infty$,
then $\r$ belongs to $L^\infty(0,T; W^{1,\infty}(\Omega))\cap
C(\Omega\times [0,T])$ instead of $C([0,T]; W^{1,\infty}(\Omega))$.
The corresponding norm is denoted by $\|\cdot\|_{M^{p,q,r}_T}$.
\end{Definition}

We remark that the condition $\int_\O P\ d\x=0$ in Definition 2.1
holds automatically if we replace $P$ by
$$P-\f{1}{|\O|}\int_\O P\ d\x$$ in \eqref{e2}.
Also, in the above definition, the space $D_{A_q}^{1-\f{1}{p},p}$
stands for some fractional domain of the Stokes operator in $L^q$
(cf. Section 2.3 in \cite{D}). Roughly, the vector-fields of
$D_{A_q}^{1-\f{1}{p},p}$ are vectors which have  $2-\f{2}{p}$
derivatives   in $L^q$, are divergence-free, and vanish on
$\partial\O$. The Besov space (for definition, see \cite{BL})
$B_{q,p}^{3(1-\f{1}{p})}$ can be regarded as the interpolation space
between $L^q$ and $W^{3,q}$, that is,
$$B_{q,p}^{3(1-\f{1}{p})}=(L^q, W^{3,q})_{1-\f{1}{p},p}.$$ Moreover,  we note that
$B_{q,p}^{3(1-\f{1}{p})}\hookrightarrow W^{1,q}$ if $p\geq
\frac{3}{2}$. By the embedding $W^{1,q}\hookrightarrow L^\infty \
\textrm {as}\  q>3$, one has $B_{q,p}^{3(1-\f{1}{p})}\hookrightarrow
L^\infty$, which will be used repeatedly in this paper.

\vspace{2mm}

The local existence will be shown by using an iterative method, and
if the initial data are sufficiently small in some suitable function
spaces, the solution is indeed global in time. More precisely, our
existence result reads as follows.

\begin{Theorem}\label{T1}
Let $\O$ be a bounded smooth domain of $\R^3$. Assume that
$\frac{3}{2}\leq p<\infty,\ 3<q\leq r\leq \infty$,
%$\frac{p}{2}-\frac{3p}{2q}<1$,
 and $\r_0\in W^{1,r}(\Omega)$ with
$\r_0\ge \check{\r}$ for some $\check{\r}>0$, $\u_0\in
D_{A_q}^{1-\f{1}{p},p}, \, \d_0\in B_{q,p}^{3(1-\f{1}{p})}$. Then,
\begin{enumerate}
\item  there exists $T_0>0$  such that system
\eqref{e2} with the initial-boundary conditions \eqref{ic}-\eqref{bc}
has a unique strong solution $(\r, \u, P, \d)\in M^{p,q,r}_{T_0}$
with $0<\check{\r}\leq \r$ on  $\O\times [0,T_0]$.

\item Moreover,  there exist a constant $\nu>0$ depending on $p, q, r, \check{\r}, \mu, \lambda, \gamma, \Omega$, and a constant $\delta>0$ depending only on $p, q, r$, such that if
$$\|\u_0\|_{D_{A_q}^{1-\frac{1}{p}, p}}+\|\d_0\|_{B_{q,p}^{3(1-\frac{1}{p})}}\leq\frac{\nu}{(1+\|\r_0\|_{W^{1,r}})^\delta}$$
holds for the initial data, then  the initial-boundary value problem \eqref{e2}-\eqref{bc} has a unique strong solution $(\r, \u, P,\d)\in
M^{p,q,r}_T$ for all $T>0$. Furthermore, denoting by $\lambda_1$ the
first eigenvalue of the Dirichlet-Laplace operator in $\Omega$, for
some constant $C$ depending on $\mu,\lambda$ and $\gamma$, we have
the following inequality for all $t\in \R^+:$
\begin{equation*}
\|(\sqrt{\r}\u)(t)\|_{L^2}+\|\nabla\d(t)\|_{L^2}\leq
Ce^{-\frac{\lambda_1}{\hat{\rho}} t}(\|\sqrt{\r_0}\u_0\|_{L^2}
+\|\nabla\d_0\|_{L^2})\left(1+(\frac{2\lambda_1}{\hat{\r}}
t)^{\frac{1}{2}}e^{\frac{\lambda_1}{\hat{\rho}} t}\right),
\end{equation*}
with $\hat{\r}=\|\rho_0\|_{L^\infty}$, %given by \eqref{5088},
and for some positive constant $K$ depending only on $\|\r_0\|_{W^{1,r}},$
$ p, q, r, \mu,
\lambda,\gamma, \check{\r}\ \textrm{and}\ \Omega$,
$$\|(\r, \u, P,\d)\|_{M_t^{p, q, r}}\leq K\left(\|\u_0\|_{D_{A_q}^{1-\frac{1}{p}, p}}+\|\d_0\|_{B_{q,p}^{3(1-\frac{1}{p})}}\right).$$
\end{enumerate}
\end{Theorem}
\bigskip

Similar to \cite{LLW},  a $\textit{weak solution}$ to \eqref{e2}
with the initial-boundary conditions \eqref{ic}-\eqref{bc} means a
quadruplet $(\tilde{\r}, \tilde{\u}, \Pi, \tilde{\bf d})$ satisfying
system \eqref{e2} in $\O\times(0,T)$ for $0<T\leq \infty$ in the
sense of distributions, i.e,
\begin{equation*}
\int\!\!\!\int_{\O\times(0,T)}(\tilde{\r}\partial_t\phi+\tilde{\r}\tilde{\u}\cdot\nabla\phi)\ d\x dt+\int_{\O}\r_0\phi(\cdot,0)\ d\x=0,
\end{equation*}
\begin{equation*}
\begin{split}
&-\int\!\!\!\int_{\O\times(0,T)}\tilde{\r}\tilde{\u}\cdot(\partial_t\phi+\tilde{\u}\cdot\nabla\phi)\ d\x dt+\int\!\!\!\int_{\O\times(0,T)}\nabla\tilde{\u}:\nabla\phi\ d\x dt\\
&=\int_\O\r_0\u_0\cdot\phi(\cdot,0)\ d\x
+\int\!\!\!\int_{\O\times(0,T)}(\nabla\d\odot\nabla\d):\nabla\phi\
d\x dt,
\end{split}
\end{equation*}
and
\begin{equation*}
\begin{split}
&-\int\!\!\!\int_{\O\times(0,T)}\tilde{\d}\cdot\partial_t\phi \ d\x dt+\int\!\!\!\int_{\O\times(0,T)}\tilde{\u}\cdot\nabla\tilde{\d}\cdot\phi \ d\x dt+\int\!\!\!\int_{\O\times(0,T)}(\nabla\tilde{\d}):(\nabla\phi)\ d\x dt\\
&=\int_\O\d_0\cdot\phi(\cdot,0)\
d\x+\int\!\!\!\int_{\O\times(0,T)}|\nabla\tilde{\d}|^2\tilde{\d}\cdot\phi\
d\x dt,
\end{split}
\end{equation*}
for all $\phi\in C_c^\infty(\O\times [0,T);\R) \ \textrm{or}\ C_c^\infty(\O\times [0,T);\R^3)$. Moreover,
$(\tilde{\u}, \tilde{\bf d})$ satisfies \eqref{bc} in the sense of
trace. In this weak formulation, the pressure $\Pi $ can be
determined as in the Navier-Stokes equations (see \cite{Galdi}).
%We state here the existence of weak solutions in Theorem ?? of \cite{??}:
%\begin{Proposition} \label{P1}
%Assume that $\u_0\in L^2(\Omega)$ and $\d_0\in H^1(\Omega)$ with
%$\d_0\mid_{\partial\Omega}\in C^2(\partial\Omega)$. Then
%the system \eqref{e2} with the initial-boundary  conditions \eqref{ic}
%\eqref{bc} has a global weak solution $({\bf v},
%\tilde{\bf d})$ such that
%$${\bf v}\in L^2(0,T; H_0^1(\Omega))\cap L^\infty(0, T; L^2(\Omega)),$$
%and
%$$\tilde{
%\d}\in L^2(0,T; H^2(\Omega))\cap L^\infty(0,T; H^1(\Omega)),$$ for
%all $T\in (0,\infty)$.
%\end{Proposition}
%\begin{Remark}\label{r}
%(1) Its uniqueness is still an open question.

%(2) Since ${\bf v}\in L^2(0,T; H^1(\Omega))\cap L^\infty(0, T; L^2(\Omega))$, it follows from the Ladyzhenskaya's inequality that ${\bf v}\in L^4(\O\times[0,T])$. Since $\nabla\tilde{\d}\in L^2(0,T;H^1(\O))$ and $|\tilde{\d}|=1$, we have
%$$\triangle\tilde{\d}\cdot\tilde{\d}+|\nabla\tilde{\d}|^2=0.$$ Hence $\nabla\tilde{\d}\in L^4(\O\times[0,T])$ and ${\bf v}\cdot\nabla{\bf v}+\nabla\cdot(\nabla\tilde{\d}\odot\nabla\tilde{\d})\in L^{\frac{4}{3}}(\O\times[0,T])$.
%\end{Remark}

\vspace{3mm}

Next, we will give a   uniqueness result. %It is similar to the result which Serrin used in the Navier-Stokes equations (\cite{JS1,JS2}), also which Lin-Liu used in the Leslie system of variable length (\cite{LL}).
Suppose $$\tilde{\r}\in L^\infty(\O\times [0,T])\cap
C(0,T;L^p(\O)),\ \forall\ p\geq 1,$$ $$\tilde{\u}\in
L^{2,\infty}(\O\times [0,T])\cap W_2^{1,0}(\O_T),\
\tilde{\r}|\tilde{\u}|^2\in L^\infty(0,T;L^1(\O)),$$
$$\tilde{\d}\in L^\infty ([0,T],H^1(\O))\cap
L^2([0,T],H^2(\O)),\ \nabla\Pi\in L^{\frac{4}{3}}(0,T;
L^{\frac{6}{5}}(\O))$$ (for all $T\in (0,\infty)$) is a global weak
solution to \eqref{e2}-\eqref{bc}. Then, we have the following
energy inequality (cf. \cite{LLW} Section 5 for the two-dimensional  homogeneous
case):
\begin{equation}\label{26}
\begin{split}
&\frac{1}{2}\int_\O(\tilde{\r}(t)|\tilde{\u}(t)|^2+|\nabla\tilde{\d}(t)|^2)
\ d\x+\int_0^t\int_\O
(|\nabla\tilde{\u}|^2+|\triangle\tilde{\d}+|\nabla\tilde{\d}|^2\tilde{\d}|^2)\
d\x d\tau\\& \leq\frac{1}{2}\int_\O(\r_0|\u_0|^2+|\nabla\d_0|^2) \
d\x,
\end{split}
\end{equation}
for all $t\in (0,\infty)$.

 However, as for the standard Navier-Stokes equations,
 the question of uniqueness in the above class of solutions remains  open.
 For the same initial-boundary conditions, the relation between
weak solutions and  strong solutions can be formulated as:

\begin{Theorem}\label{T2}
Let $\O, p, q, r$ be as in Theorem \ref{T1} and $\r_0, \u_0, \d_0$
satisfy the assumptions of Theorem \ref{T1}. Then any weak solution
to \eqref{e2}-\eqref{bc} in the above class is unique and indeed is
equal to its unique strong solution.
\end{Theorem}

Usually, the uniqueness in the above theorem is called  $\textit{weak-strong
uniqueness}$. For the similar results on the compressible
Navier-Stokes equation, we refer the  readers to \cite{EB, L}.

\bigskip

%%%%%%%%%%%%%%%%%%%%%%%%%%
\section{Maximal Regularity}

In this section, we recall a quite standard result for the transport
equation and the maximal regularities for the parabolic operator and
the non-homogeneous non-stationary Stokes operator, and prove some
$L^\infty$ estimates as well.

For $T>0$, $1<p,q<\infty$, denote
$$\mathcal{W}_{q,p}(0,T):= \big(W^{1,p}(0,T; L^q(\O))\big)^3\cap \big(L^p(0,T;W^{3,q}(\O))\big)^3.$$
Throughout this paper, $C$ stands for a generic positive constant.

We first recall a result for the transport equation (cf. Proposition
3.1 in \cite{D}):

\begin{Proposition}\label{P1}
Let $\Omega$ be a Lipschitz domain of \ $\R^3$ and ${\bf v}\in
\left(L^1(0,T; Lip)\right)^3$ be a solenoidal  vector-field such that ${\bf
v}\cdot{\bf n}=0$ on $\partial\Omega$. Let $\r_0 \in
W^{1,r}(\Omega)$ with $r\in [1,\infty]$. Then the system
\begin{equation*}
\begin{cases}
\partial_t\r+{\bf v}\cdot\nabla\r=0,\\
\r|_{t=0}=\r_0
\end{cases}
\end{equation*}
has a unique solution in $L^\infty(0,T;W^{1,\infty}(\Omega))\cap\
C([0,T];\cap_{q<\infty}W^{1,q}(\Omega))$ if $r=\infty$,  and in
$C([0,T]; W^{1,r}(\Omega))$ if $r<\infty$. Moreover, the following
estimate holds:
$$\|\r(t)\|_{W^{1,r}}\leq e^{\int_0^t\|\nabla{\bf
v}(\tau)\|_{L^\infty} d\tau}\|\r_0\|_{W^{1,r}}, \quad t\in [0,T].$$
If in addition $\r$ belongs to $L^p(\Omega)$ for some $p\in [1,\infty]$,
then $$\|\r(t)\|_{L^p}=\|\r_0\|_{L^p}, \quad t\in [0,T].$$
\end{Proposition}

We first recall the maximal regularity for the parabolic operator (cf.
Theorem 4.10.7 and Remark 4.10.9 in \cite{HA}):
\begin{Theorem}\label{T3}
Given $1<p, q<\infty$, $\omega_0\in B_{q,p}^{3(1-\f{1}{p})}$ and
$f\in \big(L^p(0,T; L^q(\R^3))\big)^3$, the Cauchy problem
\begin{equation*}
\begin{cases}
 {\omega}_t-\D\omega=f, \\
\omega|_{t=0}=\omega_0
\end{cases}
\end{equation*}
has a unique solution $\omega\in \mathcal{W}_{q,p}(0,T)$, and
$$\|\omega\|_{\mathcal{W}_{q,p}(0,T)}\le C\left(\|f\|_{L^p_T(L^q)}+\|\omega_0\|_{B_{q,p}^{3(1-\f{1}{p})}}\right),$$
where $C$ is independent of $\omega_0$, $f$ and $T$. Moreover, there
exists a positive constant $c_0$ independent of $f$ and $T$ such
that
$$\|\omega\|_{\mathcal{W}_{q,p}(0,T)}\ge c_0\sup_{t\in(0,T)}\|\omega(t)\|_{B_{q,p}^{3(1-\f{1}{p})}}.$$
\end{Theorem}

Now we recall the existence theorem (cf. Theorem 3.7 in \cite{D})
for the linear system
\begin{equation}\label{linear}
\begin{cases}
\r\partial_t\u-\mu\triangle\u+\nabla P=f,\quad \int_\O P\ d\x=0,\\
\nabla\cdot\u=0,\\
\u|_{t=0}=\u_0,\quad \u|_{\partial\O}=0.
\end{cases}
\end{equation}

\begin{Theorem}\label{T4}
Let $\O\subset\R^3$ be a bounded domain with $C^{2+\varepsilon}$
boundary, $1<p<\infty$ and $3<q\leq r\leq\infty$.
Let $\u_0\in D_{A_q}^{1-\f{1}{p},p}$ and $f\in \big(L^p(0,T;
L^q(\Omega))\big)^3$. Assume that the density $\r$ satisfies
$$0<\check{\r}\leq\r(x,t)\leq\hat{\r}<\infty, \quad (x,t)\in \O\times (0,T),$$
and for some $\beta\in (0,1]$,$$\r\in
L^\infty(0,T;W^{1,r}(\Omega))\cap C^\beta([0,T];
L^\infty(\Omega)).$$
 Then the system \eqref{linear}
has a unique solution $(\u, P)$ such that
$$\u\in C([0,T]; D_{A_q}^{1-\f{1}{p},p})\cap L^p(0,T;W^{2,q}(\O)\cap W^{1,q}_0(\O)), $$
$$\partial_t\u \in L^p(0,T;L^q(\O)),$$
and  $$P\in L^p(0,T;W^{1,q}(\Omega)).$$
Moreover, there exists some constant $C$
depending on $p, q, r \ \textrm{and}\ \Omega$ such that for all \
$t\in [0,T]$, the following inequalities hold:
\begin{equation}\label{T41}
\begin{split}
&\check{\r}^{\frac{1}{p}}\mu^{1-{\frac{1}{p}}}\|\u(t)\|_{D_{A_q}^{1-\f{1}{p},p}}+\mu\|\u\|_{L_t^p(W^{2,q})}+\check{\r}\|\partial_t\u\|_{L_t^p(L^q)}+\|P\|_{L_t^p(W^{1,q})}\\
&\le C\xi_\r^3\mathcal{B}_\r^{2+\tilde{\varsigma}}(t)e^{\frac{C \mu t \,\mathcal{C}_\r(t)}{\check{\r}d(\Omega)^2}}
\left(\check{\r}^{\frac{1}{p}}\mu^{1-{\frac{1}{p}}}\|\u_0\|_{D_{A_q}^{1-\f{1}{p},p}}+\|f\|_{L_t^p(L^q)}\right),
\end{split}
\end{equation}
and
\begin{equation}\label{T42}
\begin{split}
&\check{\r}^{\frac{1}{p}}\mu^{1-{\frac{1}{p}}}\|\u(t)\|_{D_{A_q}^{1-\f{1}{p},p}}+\|\check{\r}\partial_t\u, \mu\nabla^2\u, \nabla P\|_{L_t^p(L^q)}\\
&\le
C\left(\xi_\r^4\mathcal{B}_\r^{2+\tilde{\varsigma}}(t)\big(\check{\r}^{\frac{1}{p}}\mu^{1-{\frac{1}{p}}}\|\u_0\|_{D_{A_q}^{1-\f{1}{p},p}}+\|f\|_{L_t^p(L^q)}\big)
+\frac{\xi_\r
\mu}{d(\Omega)^2}\mathcal{C}_\r(t)\|\u\|_{L^p_t(L^q)}\right),
\end{split}
\end{equation}
where $d(\Omega)$ is the diameter of $\Omega$,\
$\xi_\r:=\hat{\r}/\check{\r}$, and
$$\mathcal{B}_\r(t):=1+d(\O)\left(\check{\r}^{-1}\|\nabla\r\|_{L_t^\infty(L^r)}\right)^{\frac{r}{r-3}},$$
$$\mathcal{C}_\r(t):=\xi_\r^{\frac{2q}{q-1}}\mathcal{B}_\r(t)^{r^\ast}+\hat{\r}\mu^{-1}d(\O)^2\xi_\r^{1+\frac{1}{\beta}}\mathcal{B}_\r(t)^{(1+\frac{1}{\beta})(2+\tilde{\varsigma})}M_\beta(t)^{\frac{1}{\beta}},$$
with $$M_\beta(t):=\check{\r}^{-1}\|\r\|_{C^{0,0}_{0,\beta}(\O\times[0,t])},$$
 and the exponents $\tilde{\varsigma}, r^\ast$ are numbers determined by $p,q,r$. %\sup_{\x\in \O}_{\tau,\tau'\in[0,t],\tau\neq\tau'}\frac{|\r()|}{}$
\end{Theorem}

\begin{Remark}
The reader can also refer to Theorem 3.7 in \cite{D} for more details
about Theorem \ref{T4}. We notice that
\eqref{T41} and \eqref{T42} do not include the estimate for
$\|\u\|_{L^p_T(L^q)}$. Indeed, since we consider only in a bounded
domain $\Omega$, then there exists a constant $C=C(q, d(\Omega))$
such that
$$\|\u\|_{W^{2,q}}\equiv \|\nabla^2\u\|_{L^q}+d(\Omega)^{-1}\|\nabla\u\|_{L^q}+d(\Omega)^{-2}\|\u\|_{L^q}\le C\|\nabla^2\u\|_{L^q},$$
whenever $\u\in W^{2,q}(\Omega)\cap W_0^{1,q}(\Omega)$\ (cf.
Proposition 2.4 in \cite{D}).
\end{Remark}

We also have the following two interpolation inequalities for the
$L^\infty$ estimates in the spatial variable (cf. Lemma 4.1 in
\cite{D}, also Lemmas 3.1,3.3 in \cite{LXWD}) which are useful in our proof.

\begin{Lemma}\label{l1}
Let $1<p,q,r,s<\infty$ satisfy
$$0<\f{p}{2}-\f{3p}{2q}<1, \quad \f{1}{s}=\f{1}{r}+\f{1}{q}.$$
 Then the  following inequalities hold:
$$\|\nabla f\|_{L^p_T (L^\infty)}\le
CT^{\f{1}{2}-\f{3}{2q}}\|f\|^{1-\theta}_{L^\infty_T(
D_{A_q}^{1-\f{1}{p},p})}\|f\|^\theta_{L^p_T (W^{2,q})},$$
$$\|\nabla f\|_{L^p_T(L^r)}\le
CT^{\f{1}{2}-\f{3}{2q}}\|f\|^{1-\theta}_{L^\infty_T(
D_{A_s}^{1-\f{1}{p},p})}\|f\|^\theta_{L^p_T(W^{2,s})},$$ for some
constant $C$ depending only on $\O, p, q, r$,  and
$$\f{1-\theta}{p}=\f{1}{2}-\f{3}{2q}.$$
\end{Lemma}

Similarly, we can prove
\begin{Lemma}\label{l2'}
Let $1<p,q,r,s<\infty$ satisfy $$0<\f{2p}{3}-\f{p}{q}<1, \quad
\f{1}{s}=\f{1}{r}+\f{1}{q}.$$ Then
$$\|\nabla f\|_{L^p_T(L^\infty)}\le
CT^{\f{2}{3}-\f{1}{q}}\|f\|^{1-\theta}_{L^\infty_T(
B_{q,p}^{3(1-\f{1}{p})})} \|f\|^\theta_{L^p_T (W^{3,q})},$$
$$\|\nabla f\|_{L^p_T(L^r)}\le
CT^{\f{2}{3}-\f{1}{q}}\|f\|^{1-\theta}_{L^\infty_T(
B_{s,p}^{3(1-\f{1}{p})})}\|f\|^\theta_{L^p_T(W^{3,s})},$$
 for some
constant $C$ depending only on $\O, p, q,r$,  and
$$\f{1-\theta}{p}=\f{2}{3}-\f{1}{q}.$$
\end{Lemma}

\begin{proof}
The proof is based on the applications of embedding and
interpolation results in \cite{BL}. First, we notice that, from
Theorem 6.4.5 in \cite{BL},
$$(B^{3-\f{3}{p}-\f{3}{q}}_{\infty,\infty}, \
B^{3-\f{3}{q}}_{\infty,\infty})_{\theta,1}=B^1_{\infty,1}\ \
\textrm{with} \ \  \f{1-\theta}{p}=\f{2}{3}-\f{1}{q},$$ and the
embedding (Theorem 6.2.4 in \cite{BL}):
$B^1_{\infty,1}\hookrightarrow W^{1,\infty}$, we get
\begin{equation}\label{5888}
\|\nabla f\|_{L^\infty}\le C\|
f\|^{\theta}_{B^{3-\f{3}{q}}_{\infty,\infty}}\|
f\|^{1-\theta}_{B^{3-\f{3}{p}-\f{3}{q}}_{\infty,\infty}}.
\end{equation}
We remark that
\begin{equation}\label{588}
B_{q,p}^{3(1-\f{1}{p})}\hookrightarrow
B_{\infty,\infty}^{3-\f{3}{p}-\f{3}{q}},\quad W^{3,q}\hookrightarrow
B^3_{q,\infty}\hookrightarrow B_{\infty,\infty}^{3-\f{3}{q}}
\end{equation}
(cf. Theorems 6.5.1 and  6.2.4 in \cite{BL}). Therefore, according
to \eqref{5888}, \eqref{588} and by H\"{o}lder's inequality, we
deduce that
\begin{equation*}
\begin{split}
\|\nabla f\|_{L^p_T( L^\infty)}&\le C\left(\int_0^T\|
f\|^{p(1-\theta)}_{B^{3-\f{3}{p}-\f{3}{q}}_{\infty,\infty}}\|
f\|^{p\theta}_{B^{3-\f{3}{q}}_{\infty,\infty}}dt\right)^{\f{1}{p}}
\le C\left(\int_0^T\|f\|^{p(1-\theta)}_{B_{q,p}^{3(1-\f{1}{p})}}\|f\|^{p\theta}_{W^{3,q}}dt\right)^{\f{1}{p}}\\
&\le CT^{\f{2}{3}-\f{1}{q}}\|f\|^{1-\theta}_{L^\infty_T(
B_{q,p}^{3(1-\f{1}{p})})}\|f\|^\theta_{L^p_T(W^{3,q})}.
\end{split}
\end{equation*}
The proof of the second inequality is based on the fact that
$$B^0_{r,1}=(B_{r,p}^{2-\frac{3}{p}-\frac{3}{q}}, B_{r,r}^{2-\frac{3}{q}})_{\theta,1}\hookrightarrow L^r\ \
\textrm{with} \ \  \f{1-\theta}{p}=\f{2}{3}-\f{1}{q}$$ (cf. Theorem
6.4.5 and Theorem 6.2.4 in \cite{BL}),  and that
$$W^{2,s}\hookrightarrow B_{r,r}^{2-\frac{3}{q}},\quad  B_{s,p}^{2-\f{3}{p}}\hookrightarrow B_{r,p}^{2-\frac{3}{p}-\frac{3}{q}},$$
(cf. Theorems 6.2.4 and 6.5.1 in \cite{BL}). In fact, by
H\"{o}lder's inequality, we have
\begin{equation*}
\begin{split}
\|\nabla f\|_{L^p_T( L^r)}&\le C\left(\int_0^T\|\nabla
f\|^{p(1-\theta)}_{B^{2-\f{3}{p}-\f{3}{q}}_{r,p}}\|\nabla
f\|^{p\theta}_{B^{2-\f{3}{q}}_{r,r}}dt\right)^{\f{1}{p}}
\le C\left(\int_0^T\|\nabla f\|^{p(1-\theta)}_{B_{s,p}^{2-\f{3}{p}}}\|\nabla f\|^{p\theta}_{W^{2,s}}dt\right)^{\f{1}{p}}\\
&\le CT^{\f{2}{3}-\f{1}{q}}\|f\|^{1-\theta}_{L^\infty_T(
B_{s,p}^{3(1-\f{1}{p})})}\|f\|^\theta_{L^p_T( W^{3,s})}.
\end{split}
\end{equation*}
\end{proof}
\begin{Lemma}\label{l3}
Let $1<p,q,r,s<\infty$ satisfy $$0<\f{p}{3}-\f{p}{q}<1,\quad
\f{1}{s}=\f{1}{r}+\f{1}{q}.$$ Then
$$\|\nabla^2 f\|_{L^p_T( L^\infty)}\le
CT^{\f{1}{3}-\f{1}{q}}\|f\|^{1-\theta}_{L^\infty_T(
B_{q,p}^{3(1-\f{1}{p})})} \|f\|^\theta_{L^p_T(W^{3,q})},$$
$$\|\nabla^2 f\|_{L^p_T( L^r)}\le
CT^{\f{1}{3}-\f{1}{q}}\|f\|^{1-\theta}_{L^\infty_T(
B_{s,p}^{3(1-\f{1}{p})})} \|f\|^\theta_{L^p_T(W^{3,s})},$$
 for some
constant $C$ depending only on $\O, p, q,r$, and
$$\f{1-\theta}{p}=\f{1}{3}-\f{1}{q}.$$
\end{Lemma}
\begin{proof}
First, we notice that (cf. Theorem 6.4.5 in
\cite{BL})
$$(B^{2-\f{3}{p}-\f{3}{q}}_{\infty,\infty}, B^{2-\f{3}{q}}_{\infty,\infty})_{\theta,1}=B^1_{\infty,1}\
\textrm{with}\ \f{1-\theta}{p}=\f{1}{3}-\f{1}{q}.$$ Hence,
\begin{equation}\label{5}
\|\nabla^2 f\|_{L^\infty}\le C\|\nabla f\|_{W^{1, \infty}}\le
C\|\nabla f\|_{B^1_{\infty,1}}\le C\|\nabla
f\|^{\theta}_{B^{2-\f{3}{q}}_{\infty,\infty}}\|\nabla
f\|^{1-\theta}_{B^{2-\f{3}{p}-\f{3}{q}}_{\infty,\infty}}.
\end{equation}
We remark that (cf. Theorems 6.2.4 and 6.5.1 in \cite{BL})
$$B_{q,p}^{3(1-\f{1}{p})}\hookrightarrow
B_{\infty,\infty}^{3-\f{3}{p}-\f{3}{q}}, \quad
W^{2,q}\hookrightarrow B_{q,\infty}^2\hookrightarrow
B_{\infty,\infty}^{2-\f{3}{q}}.$$ Thus, according to \eqref{5} and
by applying H\"{o}lder's inequality, we deduce that
\begin{equation*}
\begin{split}
\|\nabla^2 f\|_{L^p_T(L^\infty)}&\le C\left(\int_0^T\|\nabla
f\|^{p\theta}_{B^{2-\f{3}{q}}_{\infty,\infty}}\|\nabla
f\|^{p(1-\theta)}_{B^{2-\f{3}{p}-\f{3}{q}}_{\infty,\infty}}dt\right)^{\f{1}{p}}\\
&\le
C\left(\int_0^T\|f\|^{p\theta}_{W^{3,q}}\|f\|^{p(1-\theta)}_{B_{q,p}^{3(1-\f{1}{p})}}dt\right)^{\f{1}{p}}\\
&\le CT^{\f{1}{3}-\f{1}{q}}\|f\|^{1-\theta}_{L^\infty_T(
B_{q,p}^{3(1-\f{1}{p})})}\|f\|^\theta_{L^p_T(W^{3,q})}.
\end{split}
\end{equation*}
The proof of the second inequality is based on the fact that
$$B^0_{r,1}=(B_{r,p}^{1-\frac{3}{p}-\frac{3}{q}}, B_{r,r}^{1-\frac{3}{q}})_{\theta,1}\hookrightarrow L^r\ \
\textrm{with} \ \  \f{1-\theta}{p}=\f{1}{3}-\f{1}{q}$$ (cf. Theorem
6.4.5 and Theorem 6.2.4 in \cite{BL}),  and that
$$W^{1,s}\hookrightarrow B_{r,r}^{1-\frac{3}{q}},\quad  B_{s,p}^{1-\f{3}{p}}\hookrightarrow B_{r,p}^{1-\frac{3}{p}-\frac{3}{q}},$$
(cf. Theorems 6.2.4 and 6.5.1 in \cite{BL}). In fact, by
H\"{o}lder's inequality, we have
\begin{equation*}
\begin{split}
\|\nabla f\|_{L^p_T( L^r)}
&\le C\left(\int_0^T\|\nabla
f\|^{p(1-\theta)}_{B^{1-\f{3}{p}-\f{3}{q}}_{r,p}}\|\nabla
f\|^{p\theta}_{B^{1-\f{3}{q}}_{r,r}}dt\right)^{\f{1}{p}}\\
& \le C\left(\int_0^T\|\nabla f\|^{p(1-\theta)}_{B_{s,p}^{1-\f{3}{p}}}\|\nabla f\|^{p\theta}_{W^{1,s}}dt\right)^{\f{1}{p}}\\
&\le CT^{\f{1}{3}-\f{1}{q}}\|f\|^{1-\theta}_{L^\infty_T(
B_{s,p}^{2-\f{3}{p}})}\|f\|^\theta_{L^p_T( W^{2,s})}.
\end{split}
\end{equation*}
\end{proof}

\bigskip

%%%%%%%%%%%%%%%%%%%%%
\section{Local Existence}

In this section, we prove existence and uniqueness of strong
solution on a short time interval, i.e., the local strong solution
in Theorem \ref{T1}. The proof will be divided into several steps,
including constructing the approximate solutions by iteration,
obtaining the uniform estimates, showing the convergence, consistency
and uniqueness.

\subsection{Construction of approximate solutions}
We initialize the construction of approximate solutions by setting
$\r^0:=\r_0, \u^0:=\u_0$ and $\d^0:=\d_0$. Given $(\r^n,\u^n, P^n,
\d^n )$, Proposition \ref{P1}, Theorem \ref{T3} and Theorem \ref{T4}
enable us to define respectively $\r^{n+1}(\x,t)$ as the (global)
solution of the transport equation

\begin{equation}\label{r}
\begin{cases}
\partial_t\r^{n+1}+\u^n\cdot\nabla\r^{n+1}=0,\\
\r^{n+1}\mid_{t=0}=\r_0,
\end{cases}
\end{equation}
$\d^{n+1}(\x,t)$ as the (global) solution of
\begin{equation}\label{d}
\begin{cases}
\partial_t\d^{n+1}-\D\d^{n+1}=-\u^{n}\cdot\nabla
\d^{n}+|\nabla\d^n|^2\d^n, \\
\d^{n+1}\mid_{t=0}=\d_0, \quad \partial_{\bf
\nu}\d^{n+1}|_{\partial\Omega}=0,
\end{cases}
\end{equation}
and $\left((\u^{n+1}(\x,t), P^{n+1}(\x,t)\right)$ as the (global)
solution of
\begin{equation}\label{e5}
\begin{cases}
\r^{n+1}\partial_t\u^{n+1}-\D\u^{n+1}+\nabla
P^{n+1}=-\r^{n+1}\u^n\cdot\nabla\u^n-\nabla\cdot\big((\nabla\d^n)^\top\nabla\d^n\big), \\
\Dv\u^{n+1}=0,\quad \int_{\O}P^{n+1}d\x=0,\\
 \u^{n+1}|_{t=0}=\u_0, \quad
\u^{n+1}|_{\partial\O}=0.
\end{cases}
\end{equation}
An argument by induction yields a sequence $\{(\r^n, \u^n, P^n, \d^n
)\}_{n\in {\mathbb N}} \subset M^{p,q,r}_T$ for all $T>0$.

\subsection{ Uniform estimate for some small fixed time $T_\ast$}   We aim
at finding a positive time $T_\ast$ independent of $n$ for which
$\{(\r^n, \u^n, P^n, \d^n)\}_{n\in {\mathbb N}} $ is uniformly
bounded in the space $M^{p,q,r}_{T_\ast}$. Applying Proposition
\ref{P1} to \eqref{r}, we get
\begin{equation}\label{508}
\|\r^{n+1}(t)\|_{W^{1,r}}\leq
e^{\int_0^t\|\nabla\u^n(\tau)\|_{L^\infty} d\tau}\|\r_0\|_{W^{1,r}}.
\end{equation}
and
\begin{equation}\label{5088}
\min_{\x\in \bar{\Omega}}\r^{n+1}(\x,t)=\check{\r}:=\min_{\x\in
\bar{\Omega}}\r_0(\x)\ \ \textrm{and}\ \ \max_{\x\in
\bar{\Omega}}\r^{n+1}(\x,t)=\hat{\r}:=\max_{\x\in \bar{\Omega}}\r_0(\x).
\end{equation}

%In addition, the existence of a unique solution $\r^{n+1}(\x,t)$ of
%\eqref{r} follows from the method of characteristics, and for all
%time $t$, we have
%\begin{equation}\label{5088}
%\min_{\x\in \bar{\Omega}}\r^n(\x,t)=\check{\r}:=\min_{\x\in
%\bar{\Omega}}\r_0(\x)\ \ \textrm{and}\ \ \max_{\x\in
%\bar{\Omega}}\r^n(\x,t)=\hat{\r}:=\max_{\x\in \bar{\Omega}}\r_0(\x).
%\end{equation}

Since $\partial_t\r^{n+1}=-\u^n\cdot\nabla\r^{n+1}$,  then by
H\"{o}lder's inequality, we have $$\partial_t\r^{n+1}\in
L^\infty_{\textrm{loc}}(\R^+; L^s(\Omega))$$ with $s=\frac{qr}{q+r}\ (s=q \
\textrm{if}\ r=\infty)$,\ and for $t\geq 0$,\
\begin{equation}\label{500}
\|\partial_t\r^{n+1}\|_{L_t^\infty(L^s)}\leq
\|\u^n\|_{L_t^\infty(L^q)} \|\nabla\r^{n+1}\|_{L_t^\infty(L^r)}.
\end{equation}

 In order to apply Theorem \ref{T4} to \eqref{e5}, we need to prove for some $\beta\in (0,1],\  \r^{n+1}\in
C^\beta([0,T]; L^\infty(\Omega))$. Actually, noticing that by
interpolation between  $L^\infty(0,T; W^{1,r}(\Omega))$ and
$W^{1,\infty}(0,T; L^s(\Omega))$, $\r^{n+1}$ belongs to
$C^\beta([0,T]; L^\infty(\Omega))$ whenever $\beta\in
(0,\frac{1-\frac{3}{r}}{1+\frac{3}{q}})$ and it holds that
\begin{equation}\label{5000}
\|\r^{n+1}\|_{C_t^\beta(L^\infty)}\leq C\left(
\|\r^{n+1}\|_{L_t^\infty(W^{1,r})}+\|\partial_t\r^{n+1}\|_{L_t^\infty(L^s)}\right).
\end{equation}
Here we have used Young's inequality.

Hence, applying Theorem \ref{T4} to \eqref{e5} yields
\begin{equation*}
\begin{split}
&\|\u^{n+1}(t)\|_{D_{A_q}^{1-\f{1}{p},p}}+\|\u^{n+1}\|_{L_t^p(W^{2,q})}+\|\partial_t\u^{n+1}\|_{L_t^p(L^q)}+\|P^{n+1}\|_{L_t^p(W^{1,q})}\\
&\le
Ce^{Ct\psi(t)}\left(\|\u_0\|_{D_{A_q}^{1-\f{1}{p},p}}+\|\u^n\cdot\nabla\u^n+\nabla\cdot((\nabla\d^n)^\top\nabla\d^n)\|_{L_t^p(L^q)}\right),
\end{split}
\end{equation*}
where
$$\psi(t)=\left(1+\|\r^{n+1}\|_{L_t^\infty(W^{1,r})}\right)^{\gamma_0}\left(1+\|\r^{n+1}\|_{C_t^\beta(L^\infty)}^\frac{1}{\beta}\right)$$
for some positive exponent $\gamma_0$ depending only on $p, q, r,
\beta$,  and the constant $C$ depending only on $p, q, r, \check{\r},
\hat{\r}, \Omega, \beta$. Using \eqref{500} and \eqref{5000}, we get
$$\psi(t)\leq
C\left(1+\|\r^{n+1}\|_{L_t^\infty(W^{1,r})}\right)^\sigma\left(1+\|\u^n\|_{L_t^\infty(L^q)}\right)^\frac{1}{\beta},$$
where $\sigma$ depends only on $p, q,  r \ \textrm{and}\ \beta$.
 Therefore,
\begin{equation}\label{300}
\begin{split}
&\|\u^{n+1}(t)\|_{D_{A_q}^{1-\f{1}{p},p}}+\|\u^{n+1}\|_{L_t^p(W^{2,q})}+\|\partial_t\u^{n+1}\|_{L_t^p(L^q)}+\|P^{n+1}\|_{L_t^p(W^{1,q})}\\
&\le
Ce^{Ct\left(1+\|\r^{n+1}\|_{L_t^\infty(W^{1,r})}\right)^\sigma\left(1+\|\u^n\|_{L_t^\infty(L^q)}\right)^\frac{1}{\beta}}\Big(\|\u_0\|_{D_{A_q}^{1-\f{1}{p},p}}
+\|\u^n\cdot\nabla\u^n\|_{L_t^p(L^q)}\\
&\qquad+\|\nabla\cdot((\nabla\d^n)^\top\nabla\d^n)\|_{L_t^p(L^q)}\Big).
\end{split}
\end{equation}
 Applying Theorem \ref{T3} to
\eqref{d},
 we obtain
\begin{equation}\label{4}
\begin{split}
&\|\d^{n+1}(t)\|_{B_{q,p}^{3(1-\f{1}{p})}}+\|\d^{n+1}\|_{\mathcal{W}_{q,p}(0,t)}\\
&\le C\left(\|\d_0\|_{B_{q,p}^{3(1-\f{1}{p})}}+\|-\u^{n}\cdot\nabla
\d^{n}+|\nabla\d^n|^2\d^n\|_{L^p_t(L^q)}\right).
\end{split}
\end{equation}
Define
\begin{equation*}
\begin{split}
U^n(t):=&\|\u^n\|_{L^\infty_t(D_{A_q}^{1-\f{1}{p},p})}+\|\u^n\|_{L^p_t(
W^{2,q})}+\|\partial_t\u^n\|_{L^p_t(L^q)}\\
&\hskip 2cm +\|\d^n\|_{L^\infty_t(
B_{q,p}^{3(1-\f{1}{p})})}+\|\d^n\|_{\mathcal{W}_{q,p}(0,t)},
\end{split}
\end{equation*}
$$U^0\:=\|\u_0\|_{D_{A_q}^{1-\f{1}{p},p}}+\|\d_0\|_{B_{q,p}^{3(1-\f{1}{p})}},$$
$$\varrho^n(t):=\|\r^n\|_{L_t^\infty(W^{1,r})}\quad \textrm{and}\quad \varrho_0:=\|\r_0\|_{W^{1,r}}.$$
Hence, from \eqref{508}, \eqref{300} and \eqref{4}, we have the estimates for different cases as follows.
\bigskip

{\em  Case 1.} \ $\frac{2}{3}-\frac{1}{q}<\frac{1}{p}$, using Lemmas
\ref{l1}-\ref{l3}, we get,
\begin{equation}\label{60}
\begin{split}
\varrho^{n+1}(t)\le
e^{t^{1-\frac{1}{p}}\|\nabla\u^n\|_{L_t^p(L^\infty)}}\|\r_0\|_{W^{1,r}}\leq
\varrho_0e^{Ct^{\frac{3}{2}-\frac{1}{p}-\frac{3}{2q}}U^n(t)},
\end{split}
\end{equation}
\begin{equation*}
\begin{split}
&\|\u^{n+1}(t)\|_{D_{A_q}^{1-\f{1}{p},p}}+\|\u^{n+1}\|_{L_t^p(W^{2,q})}+\|\partial_t\u^{n+1}\|_{L_t^p(L^q)}+\|P^{n+1}\|_{L_t^p(W^{1,q})}\\
&\le
Ce^{Ct\left(1+\|\r^{n+1}\|_{L_t^\infty(W^{1,r})}\right)^\sigma\left(1+\|\u^n\|_{L_t^\infty(L^q)}\right)^\frac{1}{\beta}}\Big(\|\u_0\|_{D_{A_q}^{1-\f{1}{p},p}}
+\|\u^n\|_{L_t^\infty(L^q)}\|\nabla\u^n\|_{L_t^p(L^\infty)}\\
&\qquad+\|\nabla\d^n\|_{L_t^\infty(L^q)}\|\triangle\d^n\|_{L_t^p(L^\infty)}\Big),
\end{split}
\end{equation*}
\begin{equation*}
\begin{split}
&\|\d^{n+1}(t)\|_{B_{q,p}^{3(1-\f{1}{p})}}+\|\d^{n+1}\|_{\mathcal{W}_{q,p}(0,t)}\\
&\leq C\Big(\|\d_0\|_{B_{q,p}^{3(1-\f{1}{p})}}+\|\u^n\|_{L^\infty_t(
L^q)}\|\nabla\d^n\|_{L^p_t(L^\infty)}+\|\d^n\|_{L^\infty_t(
L^\infty)}\|\nabla\d^n\|_{L^\infty_t(L^q)}\|\nabla\d^n\|_{L^p_t(
L^\infty)}\Big),
\end{split}
\end{equation*}
and
\begin{equation}\label{6}
\begin{split}
U^{n+1}(t)&\le
Ce^{Ct\big(1+\varrho^{n+1}(t)\big)^\sigma\big(1+U^n(t)\big)^{\frac{1}{\beta}}}\Big(U^0+(t^{\frac{1}{2}-\frac{3}{2q}}+t^{\frac{1}{3}-\frac{1}{q}}+t^{\frac{2}{3}-\frac{1}{q}})\big(U^n(t)\big)^2
\\
&\qquad+t^{\frac{2}{3}-\frac{1}{q}}\big(U^n(t)\big)^3\Big).
\end{split}
\end{equation}
%Inserting \eqref{60} in \eqref{6} yields
%\begin{equation*}
%\begin{split}
%&U^{n+1}(t)\leq C
%e^{Ct\big(1+U^n(t)\big)^{\frac{1}{\beta}}(1+\varrho_0)^\sigma
%e^{C\sigma
%t^{\frac{3}{2}-\frac{1}{p}-\frac{3}{2q}}U^n(t)}}\Big(U^0+(t^{\frac{1}{2}-\frac{3}{2q}}+t^{\frac{1}{3}-\frac{1}{q}}+t^{\frac{2}{3}-\frac{1}{q}})\big(U^n(t)\big)^2\\
%&\qquad\qquad\quad+t^{\frac{2}{3}-\frac{1}{q}}\big(U^n(t)\big)^3\Big).
%\end{split}
%\end{equation*}
 Assuming
that $t$ is sufficiently small so that
\begin{equation}\label{60'}
Ct^{\frac{3}{2}-\frac{1}{p}-\frac{3}{2q}}U^n(t)\leq\ln2,
\end{equation}
 we get
from \eqref{60} that
\begin{equation}\label{rr}
\varrho^{n+1}(t)\le2\varrho_0,
\end{equation}
 and from \eqref{6} that
\begin{equation*}%\label{60''''}
\begin{split}
U^{n+1}(t)&\leq C e^{2^\sigma C
t\big(1+U^n(t)\big)^{\frac{1}{\beta}}(1+\varrho_0)^\sigma}\Big(U^0+(t^{\frac{1}{2}-\frac{3}{2q}}
+t^{\frac{1}{3}-\frac{1}{q}}+t^{\frac{2}{3}-\frac{1}{q}})\big(U^n(t)\big)^2\\
&\qquad+t^{\frac{2}{3}-\frac{1}{q}}\big(U^n(t)\big)^3\Big).
\end{split}
\end{equation*}

{\em Case 2.} \ $\frac{2}{3}-\frac{1}{q}=\frac{1}{p}$, then $B_{q,p}^{3(1-\frac{1}{p})}\hookrightarrow W^{1,\infty-}$ with $\infty-$ denoting any positive number large enough (but not $\infty$), using Lemmas
\ref{l1} and \ref{l3}, we get
\begin{equation*}
\begin{split}
&\|\d^{n+1}(t)\|_{B_{q,p}^{3(1-\f{1}{p})}}+\|\d^{n+1}\|_{\mathcal{W}_{q,p}(0,t)}\\
&\leq
C\Big(\|\d_0\|_{B_{q,p}^{3(1-\f{1}{p})}}+t^{\frac{1}{p}}\|\u^n\|_{L^\infty_t(
L^{q+})}\|\nabla\d^n\|_{L^\infty_t(L^{\infty-})}\\
&\qquad\quad+t^{\frac{1}{p}}\|\d^n\|_{L^\infty_t(
L^\infty)}\|\nabla\d^n\|_{L^\infty_t(L^{q+})}\|\nabla\d^n\|_{L^\infty_t(
L^{\infty-})}\Big),
\end{split}
\end{equation*}
and
\begin{equation*}%\label{6'}
\begin{split}
U^{n+1}(t)&\le
Ce^{Ct\big(1+\varrho^{n+1}(t)\big)^\sigma\big(1+U^n(t)\big)^{\frac{1}{\beta}}}\Big(U^0+(t^{\frac{1}{2}-\frac{3}{2q}}+t^{\frac{1}{3}-\frac{1}{q}}+t^{\frac{1}{p}})\big(U^n(t)\big)^2\\
&\qquad+t^{\frac{1}{p}}\big(U^n(t)\big)^3\Big).
\end{split}
\end{equation*}
%Inserting \eqref{60} in \eqref{6'} yields
%\begin{equation}\label{60''}
%\begin{split}
%&U^{n+1}(t)\leq C
%e^{Ct\big(1+U^n(t)\big)^{\frac{1}{\beta}}(1+\varrho_0)^\sigma
%e^{C\sigma
%t^{\frac{3}{2}-\frac{1}{p}-\frac{3}{2q}}U^n(t)}}\Big(U^0+(t^{\frac{1}{2}-\frac{3}{2q}}+t^{\frac{1}{3}-\frac{1}{q}}+t^{\frac{1}{p}})\big(U^n(t)\big)^2\\
%&\qquad\qquad\quad+t^{\frac{1}{p}}\big(U^n(t)\big)^3\Big).
%\end{split}
%\end{equation}
 Assuming that $t$ is sufficiently small as that is in \eqref{60'}, we get
\eqref{rr} and
\begin{equation}\label{l}
\begin{split}
&U^{n+1}(t)\\
&\leq C e^{2^\sigma C
t\big(1+U^n(t)\big)^{\frac{1}{\beta}}(1+\varrho_0)^\sigma}\left(U^0+(t^{\frac{1}{2}-\frac{3}{2q}}
+t^{\frac{1}{3}-\frac{1}{q}}+t^{\frac{1}{p}})\big(U^n(t)\big)^2+t^{\frac{1}{p}}\big(U^n(t)\big)^3\right).
\end{split}
\end{equation}

{\em  Case 3.} \ $\frac{2}{3}-\frac{1}{q}>\frac{1}{p}>\frac{1}{2}-\frac{3}{2q}$, then $B_{q,p}^{3(1-\frac{1}{p})}\hookrightarrow W^{1,\infty}$, using Lemmas
\ref{l1} and \ref{l3}, we get
\begin{equation*}
\begin{split}
&\|\d^{n+1}(t)\|_{B_{q,p}^{3(1-\f{1}{p})}}+\|\d^{n+1}\|_{\mathcal{W}_{q,p}(0,t)}\\
&\leq
C\Big(\|\d_0\|_{B_{q,p}^{3(1-\f{1}{p})}}+t^{\frac{1}{p}}\|\u^n\|_{L^\infty_t(
L^{q})}\|\nabla\d^n\|_{L^\infty_t(L^{\infty})}\\
&\qquad\quad+t^{\frac{1}{p}}\|\d^n\|_{L^\infty_t(
L^\infty)}\|\nabla\d^n\|_{L^\infty_t(L^{q})}\|\nabla\d^n\|_{L^\infty_t(
L^{\infty})}\big)\Big),
\end{split}
\end{equation*}
and again, by choosing  $t$ sufficiently small as that is in
\eqref{60'}, \eqref{rr} and \eqref{l} follow.

{\em Case 4.} \ $\frac{1}{2}-\frac{3}{2q}=\frac{1}{p}$, then $D_{A_q}^{1-\frac{1}{p},p}\hookrightarrow W^{1,\infty-}$ and $B_{q,p}^{3(1-\frac{1}{p})}\hookrightarrow W^{1,\infty}$,
using Lemma \ref{l3}, we get,
\begin{equation}\label{60'''}
\begin{split}
\varrho^{n+1}(t)\leq \varrho_0e^{Ct^{1-\frac{1}{p}}U^n(t)},
\end{split}
\end{equation}
and
\begin{equation*}
\begin{split}
&\|\u^{n+1}(t)\|_{D_{A_q}^{1-\f{1}{p},p}}+\|\u^{n+1}\|_{L_t^p(W^{2,q})}+\|\partial_t\u^{n+1}\|_{L_t^p(L^q)}+\|P^{n+1}\|_{L_t^p(W^{1,q})}\\
&\le
Ce^{Ct\left(1+\|\r^{n+1}\|_{L_t^\infty(W^{1,r})}\right)^\sigma\left(1+\|\u^n\|_{L_t^\infty(L^q)}\right)^\frac{1}{\beta}}\Big(\|\u_0\|_{D_{A_q}^{1-\f{1}{p},p}}
+t^{\frac{1}{p}}\|\u^n\|_{L_t^\infty(L^{q+})}\|\nabla\u^n\|_{L_t^\infty(L^{\infty-})}\\
&\qquad+\|\nabla\d^n\|_{L_t^\infty(L^q)}\|\triangle\d^n\|_{L_t^p(L^\infty)}\Big),
\end{split}
\end{equation*}
\begin{equation}\label{6'''}
\begin{split}
&U^{n+1}(t)\\ &\le
Ce^{Ct\big(1+\varrho^{n+1}(t)\big)^\sigma\big(1+U^n(t)\big)^{\frac{1}{\beta}}}\left(U^0+(t^{\frac{1}{3}-\frac{1}{q}}+2t^{\frac{1}{p}})\big(U^n(t)\big)^2
+t^{\frac{1}{p}}\big(U^n(t)\big)^3\right).
\end{split}
\end{equation}
%Inserting \eqref{60'''} in \eqref{6'''} yields
%\begin{equation}\label{ll}
%\begin{split}
%&U^{n+1}(t)\leq C
%e^{Ct\big(1+U^n(t)\big)^{\frac{1}{\beta}}(1+\varrho_0)^\sigma
%e^{C\sigma
%t^{1-\frac{1}{p}}U^n(t)}}\left(U^0+(t^{\frac{1}{3}-\frac{1}{q}}+2t^{\frac{1}{p}})\big(U^n(t)\big)^2+t^{\frac{1}{p}}\big(U^n(t)\big)^3\right).
%\end{split}
%\end{equation}
 Assuming
that $t$ is sufficiently small so that
\begin{equation}\label{llll}
Ct^{1-\frac{1}{p}}U^n(t)\leq\ln2,
\end{equation}
 we get from \eqref{60'''} that \eqref{rr} holds and
\begin{equation}\label{lll}
\begin{split}
U^{n+1}(t)&\leq C e^{2^\sigma C
t\big(1+U^n(t)\big)^{\frac{1}{\beta}}(1+\varrho_0)^\sigma}\Big(U^0+(
t^{\frac{1}{3}-\frac{1}{q}}+2t^{\frac{1}{p}})\big(U^n(t)\big)^2+t^{\frac{1}{p}}\big(U^n(t)\big)^3\Big).
\end{split}
\end{equation}

{\em Case 5.} \ $\frac{1}{2}-\frac{3}{2q}>\frac{1}{p}>\frac{1}{3}-\frac{1}{q}$, then $D_{A_q}^{1-\frac{1}{p},p}\hookrightarrow W^{1,\infty}$ and $B_{q,p}^{3(1-\frac{1}{p})}\hookrightarrow W^{1,\infty}$,
using Lemma \ref{l3}, we get,
\begin{equation*}
\begin{split}
&\|\u^{n+1}(t)\|_{D_{A_q}^{1-\f{1}{p},p}}+\|\u^{n+1}\|_{L_t^p(W^{2,q})}+\|\partial_t\u^{n+1}\|_{L_t^p(L^q)}+\|P^{n+1}\|_{L_t^p(W^{1,q})}\\
&\le
Ce^{Ct\left(1+\|\r^{n+1}\|_{L_t^\infty(W^{1,r})}\right)^\sigma\left(1+\|\u^n\|_{L_t^\infty(L^q)}\right)^\frac{1}{\beta}}\Big(\|\u_0\|_{D_{A_q}^{1-\f{1}{p},p}}
+t^{\frac{1}{p}}\|\u^n\|_{L_t^\infty(L^q)}\|\nabla\u^n\|_{L_t^\infty(L^\infty)}\\
&\qquad+\|\nabla\d^n\|_{L_t^\infty(L^q)}\|\triangle\d^n\|_{L_t^p(L^\infty)}\Big),
\end{split}
\end{equation*}
and \eqref{6'''} follows. Moreover, by choosing $t$ sufficiently small as
  in \eqref{llll}, \eqref{rr} and \eqref{lll} follow.

{\em Case 6.} \
$\frac{1}{3}-\frac{1}{q}=\frac{1}{p}$, then
$D_{A_q}^{1-\frac{1}{p},p}\hookrightarrow W^{1,\infty}$ and
$B_{q,p}^{3(1-\frac{1}{p})}\hookrightarrow W^{2,\infty-}$, we get
\begin{equation*}
\begin{split}
&\|\u^{n+1}(t)\|_{D_{A_q}^{1-\f{1}{p},p}}+\|\u^{n+1}\|_{L_t^p(W^{2,q})}+\|\partial_t\u^{n+1}\|_{L_t^p(L^q)}+\|P^{n+1}\|_{L_t^p(W^{1,q})}\\
&\le
Ce^{Ct\left(1+\|\r^{n+1}\|_{L_t^\infty(W^{1,r})}\right)^\sigma\left(1+\|\u^n\|_{L_t^\infty(L^q)}\right)^\frac{1}{\beta}}\Big(\|\u_0\|_{D_{A_q}^{1-\f{1}{p},p}}
+t^{\frac{1}{p}}\|\u^n\|_{L_t^\infty(L^q)}\|\nabla\u^n\|_{L_t^\infty(L^\infty)}\\
&\qquad+t^{\frac{1}{p}}\|\nabla\d^n\|_{L_t^\infty(L^{q+})}\|\triangle\d^n\|_{L_t^p(L^{\infty-})}\Big),
\end{split}
\end{equation*}
and
\begin{equation}\label{6'''''}
\begin{split}
U^{n+1}(t)&\le
Ce^{Ct\big(1+\varrho^{n+1}(t)\big)^\sigma\big(1+U^n(t)\big)^{\frac{1}{\beta}}}\left(U^0+3t^{\frac{1}{p}}\big(U^n(t)\big)^2
+t^{\frac{1}{p}}\big(U^n(t)\big)^3\right).
\end{split}
\end{equation}
%Inserting \eqref{60'''} in \eqref{6'''''} yields
%\begin{equation}\label{rrr}
%\begin{split}
%&U^{n+1}(t)\leq C
%e^{Ct\big(1+U^n(t)\big)^{\frac{1}{\beta}}(1+\varrho_0)^\sigma
%e^{C\sigma
%t^{1-\frac{1}{p}}U^n(t)}}\left(U^0+3t^{\frac{1}{p}}\big(U^n(t)\big)^2+t^{\frac{1}{p}}\big(U^n(t)\big)^3\right).
%\end{split}
%\end{equation}
 Assuming that $t$ is sufficiently small as   in \eqref{llll}, we get
\eqref{rr} and
\begin{equation}\label{lllll}
\begin{split}
U^{n+1}(t)&\leq C e^{2^\sigma C
t\big(1+U^n(t)\big)^{\frac{1}{\beta}}(1+\varrho_0)^\sigma}\Big(U^0+
3t^{\frac{1}{p}}\big(U^n(t)\big)^2+t^{\frac{1}{p}}\big(U^n(t)\big)^3\Big).
\end{split}
\end{equation}

{\em  Case 7.} \
$\frac{1}{3}-\frac{1}{q}>\frac{1}{p}$, then
$D_{A_q}^{1-\frac{1}{p},p}\hookrightarrow W^{1,\infty}$ and
$B_{q,p}^{3(1-\frac{1}{p})}\hookrightarrow W^{2,\infty}$, we get
\begin{equation*}
\begin{split}
&\|\u^{n+1}(t)\|_{D_{A_q}^{1-\f{1}{p},p}}+\|\u^{n+1}\|_{L_t^p(W^{2,q})}+\|\partial_t\u^{n+1}\|_{L_t^p(L^q)}+\|P^{n+1}\|_{L_t^p(W^{1,q})}\\
&\le
Ce^{Ct\left(1+\|\r^{n+1}\|_{L_t^\infty(W^{1,r})}\right)^\sigma\left(1+\|\u^n\|_{L_t^\infty(L^q)}\right)^\frac{1}{\beta}}\Big(\|\u_0\|_{D_{A_q}^{1-\f{1}{p},p}}
+t^{\frac{1}{p}}\|\u^n\|_{L_t^\infty(L^q)}\|\nabla\u^n\|_{L_t^\infty(L^\infty)}\\
&\qquad+t^{\frac{1}{p}}\|\nabla\d^n\|_{L_t^\infty(L^q)}\|\triangle\d^n\|_{L_t^p(L^{\infty})}\Big),
\end{split}
\end{equation*}
and \eqref{6'''''} follows. Assuming that $t$ is sufficiently small
as   in \eqref{llll}, then \eqref{rr} and \eqref{lllll} hold.

 Hence, for Cases 1, 2, and 3, if we assume that $U^n(t)\le 4CU^0$ on $[0,T_\ast]$ with
\begin{equation}\label{TTT1}
\begin{split}
T_\ast=&\min\left\{
\left(\frac{\ln2}{4C^2U^0}\right)^{\frac{3pq-2q-3p}{2pq}},
\frac{\ln2}{2^\sigma C(1+\varrho_0)^\sigma
\left(1+4CU^0\right)^{\frac{1}{\beta}}} , \right.\\
&\qquad\qquad \left. \left(\frac{1}{16C^2U^0(3+4CU^0)}\right)^{\frac{3q}{q-3}}\right\}\leq1
\end{split}
\end{equation}
or
\begin{equation}\label{TTT2}
\begin{split}
 1<T_\ast=&\min\left\{
\left(\frac{\ln2}{4C^2U^0}\right)^{\frac{3pq-2q-3p}{2pq}},
\frac{\ln2}{2^\sigma C(1+\varrho_0)^\sigma
\left(1+4CU^0\right)^{\frac{1}{\beta}}} ,\right.\\
&\qquad\qquad \left. \left(\frac{1}{16C^2U^0(3+4CU^0)}\right)^{\max\{p,\frac{3q}{2q-3}\}}\right\},
\end{split}
\end{equation}
and, for Cases 4, 5, 6, and 7, if we assume that $U^n(t)\le 4CU^0$
on $[0,T_\ast]$ with
\begin{equation}\label{TTT1'}
\begin{split}
T_\ast=&\min\left\{
\left(\frac{\ln2}{4C^2U^0}\right)^{\frac{p}{p-1}},
\frac{\ln2}{2^\sigma C(1+\varrho_0)^\sigma
\left(1+4CU^0\right)^{\frac{1}{\beta}}} ,\right.\\
&\qquad\qquad \left. \left(\frac{1}{16C^2U^0(3+4CU^0)}\right)^{\max\{p,\frac{3q}{q-3}\}}\right\}\leq1
\end{split}
\end{equation}
or
\begin{equation}\label{TTT2'}
\begin{split}
1<T_\ast=&\min\left\{
\left(\frac{\ln2}{4C^2U^0}\right)^{\frac{p}{p-1}},
\frac{\ln2}{2^\sigma C(1+\varrho_0)^\sigma
\left(1+4CU^0\right)^{\frac{1}{\beta}}} , \right.\\
&\qquad\qquad \left. \left(\frac{1}{16C^2U^0(3+4CU^0)}\right)^p\right\},
\end{split}
\end{equation}
 then a direct computation yields
\begin{equation*}
U^{n+1}(t)\le 4CU^0 \ \ \textrm{on}\ \ [0,T_\ast].
\end{equation*}
Coming back to \eqref{60}, we conclude that the sequence $\{(\r^n,
\u^n, P^n, \d^n)\}$ is uniformly bounded in $M^{p,q,r}_{T_\ast}$.
More precisely, we have proved the following estimates:

\begin{Lemma}\label{41}
For all $t\in [0,T_\ast]$ with $T_\ast$ satisfying   \eqref{TTT1} or \eqref{TTT2} for Cases 1-3,
and \eqref{TTT1'} or \eqref{TTT2'} for Cases 4-7,
\begin{equation}\label{7}
\varrho^n(t)\le 2\varrho_0 \ \ \textrm{and}\ \ U^n(t)\le 4CU^0.
\end{equation}
\end{Lemma}

\subsection{Convergence of the approximate sequence}
\begin{Lemma}
There exists $T_0$ such that $\{(\r^n, \u^n, P^n,
\d^n)\}_{n=1}^\infty$ is a Cauchy sequence in $M_{T_0}^{p,s,r}$ and
thus converges.
\end{Lemma}
\begin{proof}
Let
$$\bar{\r}^n:=\r^{n+1}-\r^n, \quad \bar{\u}^n:=\u^{n+1}-\u^n, \quad \bar{P}^n:=P^{n+1}-P^n, \quad \bar{\F}^n:=\F^{n+1}-\F^n.$$
Define
\begin{equation*}
\begin{split}
 \bar{U}^n(t):=&\|\bar{\u}^n\|_{L^\infty_t(
D_{A_s}^{1-\f{1}{p},p})}+\|\bar{\u}^n\|_{L^p_t(
W^{2,s})}+\|\partial_t\bar{\u}^n\|_{L^p_t(L^s)}\\
&+\|\nabla\bar{P}^n\|_{L^p_t(L^s)}+\|\bar{\d}^n\|_{L^\infty_t(
B_{s,p}^{3(1-\f{1}{p})})}+\|\bar{\d}^n\|_{\mathcal{W}_{s,p}(0,t)}.
\end{split}
\end{equation*}
It is easy to verify that $(\bar{\r}^n, \bar{\u}^n, \bar{P}^n,
\bar{\d}^n)$ satisfies
\begin{equation*}
\begin{cases}
\partial_t\bar{\r}^n+\u^n\cdot\nabla\bar{\r}^n=-\bar{\u}^{n-1}\cdot\nabla\r^n, \\
\r^{n+1}\partial_t\bar{\u}^n-\triangle\bar{\u}^n+\nabla
\bar{P}^n=-\bar{\r}^n(\partial_t\u^n+\u^n\cdot\nabla\u^n)-\r^n(\u^n\cdot\nabla\bar{\u}^{n-1}+\bar{\u}^{n-1}\cdot\nabla\u^{n-1})\\
\qquad\qquad\qquad\qquad\qquad\qquad-\nabla\cdot\left((\nabla\bar{\d}^{n-1})^\top\nabla\d^n\right)
-\nabla\cdot\left((\nabla\d^{n-1})^\top\nabla\bar{\d}^{n-1}\right),\\
\partial_t\bar{\F}^{n}-\D\bar{\F}^{n}=-\bar{\u}^{n-1}\cdot\nabla \F^{n}-\u^{n-1}\cdot\nabla\bar{\F}^{n-1}+|\nabla\d^n|^2\bar{\d}^{n-1}\\
\qquad\qquad\qquad\quad
+\left((\nabla\d^n+\nabla\d^{n-1}):\nabla\bar{\d}^{n-1}\right)\d^{n-1},\\
\Dv\bar{\u}^{n}=0, \quad \int_\O \bar{P}^n d\x=0
\end{cases}
\end{equation*}
with the initial-boundary conditions:
$$(\bar{\rho}^n, \bar{\u}^n, \bar{\d}^n)|_{t=0}=(0,0,0),$$
$$(\bar{\u}^n,\partial_\nu\bar{\d}^n)|_{\partial\O}=(0,0).$$
Applying Theorem \ref{T3} to
\begin{equation*}
\begin{cases}
\partial_
t\bar{\d}^n-\triangle\bar{\d}^n=-\bar{\u}^{n-1}\cdot\nabla\d^n-\u^{n-1}\cdot\nabla\bar{\d}^{n-1}+|\nabla\d^n|^2\bar{\d}^{n-1}\\
\qquad\qquad\qquad\quad+\left((\nabla\d^n+\nabla\d^{n-1}):\nabla\bar{\d}^{n-1}\right)\d^{n-1},\\
\bar{\d}^n|_{t=0}=0,\quad
\partial_\nu\bar{\d}^n|_{\partial\O}=0,
\end{cases}
\end{equation*}
and applying Theorem \ref{T4} to
\begin{equation*}
\begin{cases}
\r^{n+1}\partial_t\bar{\u}^n-\triangle\bar{\u}^n+\nabla
\bar{P}^n=-\bar{\r}^n(\partial_t\u^n+\u^n\cdot\nabla\u^n)-\r^n(\u^n\cdot\nabla\bar{\u}^{n-1}+\bar{\u}^{n-1}\cdot\nabla\u^{n-1})\\
\qquad\qquad\qquad\qquad\qquad\qquad-\nabla\cdot\left((\nabla\bar{\d}^{n-1})^\top\nabla\d^n\right)-\nabla\cdot\left((\nabla\d^{n-1})^\top\nabla\bar{\d}^{n-1}\right),\\
\nabla\cdot\bar{\u}^{n}=0, \quad \int_\O \bar{P}^n d\x=0,\\
\bar{\u}^n|_{t=0}=0,\quad \bar{\u}^n|_{\partial\O}=0,
\end{cases}
\end{equation*}
 we have
\begin{equation*}
\begin{split}
\bar{U}^n(t)&\leq
C\Big(\|\bar{\r}^n\big(\partial_t\u^n+\u^n\cdot\nabla\u^n\big)\|_{L_t^p(L^s)}+\|\u^n\cdot\nabla\bar{\u}^{n-1}
+\bar{\u}^{n-1}\cdot\nabla\u^{n-1}\|_{L_t^p(L^s)}\\
&\qquad\quad+\|\nabla\cdot\big((\nabla\bar{\d}^{n-1})^\top\nabla\d^n\big)+\nabla\cdot\big((\nabla\d^{n-1})^\top\nabla\bar{\d}^{n-1}\big)\|_{L_t^p(L^s)}\\
&\qquad\quad+\|\bar{\u}^{n-1}\cdot\nabla\d^n
+\u^{n-1}\cdot\nabla\bar{\d}^{n-1}\|_{L_t^p(L^s)}\\
&\qquad\quad+\||\nabla\d^n|^2\bar{\d}^{n-1}+\left((\nabla\d^n+\nabla\d^{n-1}):\nabla\bar{\d}^{n-1}\right)\d^{n-1}\|_{L_t^p(L^s)}\Big)\\
&\leq C
\Big(\|\bar{\r}^n\|_{L_t^\infty(L^r)}\big(\|\partial_t\u^n\|_{L_t^p(L^q)}
+\|\u^n\|_{L_t^\infty(L^q)}\|\nabla\u^n\|_{L_t^p(L^\infty)}\big)\\
&\qquad\quad+\|\u^n\|_{L_t^\infty(L^q)}\|\nabla\bar{\u}^{n-1}
\|_{L_t^p(L^r)}+\|\bar{\u}^{n-1}\|_{L_t^\infty(L^s)}\|\nabla\u^{n-1}\|_{L_t^p(L^\infty)}\\
&\qquad\quad+\|\nabla\d^n\|_{L_t^\infty(L^q)}\|\triangle\bar{\d}^{n-1}\|_{L_t^p(L^r)}+\|\nabla\bar{\d}^{n-1}\|_{L_t^\infty(L^s)}\|\triangle\d^n\|_{L_t^p(L^\infty)}\\
&\qquad\quad+\|\nabla\bar{\d}^{n-1}\|_{L_t^\infty(L^s)}\|\triangle\d^{n-1}
\|_{L_t^p(L^\infty)}+\|\nabla\d^{n-1}\|_{L_t^\infty(L^q)}\|\triangle\bar{\d}^{n-1}\|_{L_t^p(L^r)}\\
&\qquad\quad+\|\bar{\u}^{n-1}\|_{L_t^\infty(L^s)}\|\nabla\d^n
\|_{L_t^p(L^\infty)}+\|\u^{n-1}\|_{L_t^\infty(L^q)}\|\nabla\bar{\d}^{n-1}\|_{L_t^p(L^r)}\\
&\qquad\quad+\|\nabla\d^n\|_{L_t^\infty(L^q)}\|\nabla\d^n\|_{L_t^p(L^r)}\|\bar{\d}^{n-1}
\|_{L_t^\infty(L^\infty)}\\
&\qquad\quad+\|\nabla\bar{\d}^{n-1}\|_{L_t^\infty(L^s)}\big(\|\nabla\d^n\|_{L_t^p(L^\infty)}+\|\nabla\d^{n-1}\|_{L_t^p(L^\infty)}\big)\|\d^{n-1}\|_{L_t^\infty(L^\infty)}\Big).
\end{split}
\end{equation*}
Note that if $\frac{1}{2}-\frac{3}{2q}<\frac{1}{p}$, using Young's
inequality, Lemma \ref{l1} yields $$\|\nabla
\bar{\u}^{n-1}\|_{L^p_t(L^r)}\le
Ct^{\f{1}{2}-\f{3}{2q}}\left(\|\bar{\u}^{n-1}\|^{1-\theta}_{L^\infty_t(
D_{A_s}^{1-\f{1}{p},p})}+\|\bar{\u}^{n-1}\|^\theta_{L^p_t(W^{2,s})}\right).$$
 If $\frac{1}{2}-\frac{3}{2q}>\frac{1}{p}$, we have $D_{A_s}^{1-\frac{1}{p},p}\hookrightarrow
 W^{1,r}$ so that the above inequality holds with the power of $t$ replaced by $t^{\frac{1}{p}}$.
 As for $\frac{1}{2}-\frac{3}{2q}=\frac{1}{p}$,  since
 $$\|\u^n\cdot\nabla\bar{\u}^{n-1}\|_{L_t^p(L^s)}\leq\|\u^n\|_{L_t^\infty(L^{q+})}\|\nabla\bar{\u}^{n-1}
\|_{L_t^p(L^{r-})}$$  and $D_{A_s}^{1-\frac{1}{p},p}\hookrightarrow
 L^{q+}$, $D_{A_s}^{1-\frac{1}{p},p}\hookrightarrow
 W^{1,r-}$ with $q+$ (resp. $r-$) slightly greater (resp. smaller) than $q$ (resp. $r$), we still have
 $$\|\u^n\cdot\nabla\bar{\u}^{n-1}\|_{L_t^p(L^s)}\leq t^{\frac{1}{p}}\|\u^n\|_{L_t^\infty(D_{A_s}^{1-\frac{1}{p},p})}\|\bar{\u}^{n-1}
\|_{L_t^\infty(D_{A_s}^{1-\frac{1}{p},p})}.$$
  The other terms such as $\|\triangle\bar{\d}^{n-1}\|_{L_t^p(L^r)}, \|\nabla\bar{\d}^{n-1}\|_{L_t^p(L^r)}$
 and $\|\triangle\d^n\|_{L_t^p(L^r)}$ may be handled via the
similar technique by using the Besov space $B_{s,p}^{3(1-\f{1}{p})}$.

To simplify the presentation, assume from now that
$\frac{2}{3}-\frac{1}{q}<\frac{1}{p}$ so that Lemmas
\ref{l1}-\ref{l3} can be applied.
Otherwise, according to the arguments above, we would get
$t^{\frac{1}{p}}$ instead of $t^{\frac{2}{3}-\frac{1}{q}}$,
$t^{\frac{1}{2}-\frac{3}{2q}}$, or $t^{\frac{1}{3}-\frac{1}{q}}$
below once their exponents are greater than or equal to
$\frac{1}{p}$.

Hence, for all $t\in[0,T_\ast]$, taking advantage of \eqref{7}, the
embedding $$B_{s,p}^{3(1-\f{1}{p})}\hookrightarrow W^{1,s}\
\textrm{as}\ \ s=\frac{qr}{q+r}>\frac{3}{2}$$ and Lemmas
\ref{l1}-\ref{l3}, we get
\begin{equation}\label{8}
\begin{split}
\bar{U}^n(t)\leq &
C\Big(\|\bar{\r}^n\|_{L_t^\infty(L^r)}+\|\nabla\bar{\u}^{n-1}
\|_{L_t^p(L^r)}+t^{\frac{1}{2}-\frac{3}{2q}}\|\bar{\u}^{n-1}\|_{L_t^\infty(L^s)}+\|\triangle\bar{\d}^{n-1}\|_{L_t^p(L^r)}\\
&\qquad+t^{\frac{1}{3}-\frac{1}{q}}\|\nabla\bar{\d}^{n-1}\|_{L_t^\infty(L^s)}
+t^{\frac{2}{3}-\frac{1}{q}}\|\bar{\u}^{n-1}\|_{L_t^\infty(L^s)}+\|\nabla\bar{\d}^{n-1}\|_{L_t^p(L^r)}\\
&\qquad+t^{\frac{2}{3}-\frac{1}{q}}\|\bar{\d}^{n-1}\|_{L_t^\infty(L^\infty)}+
t^{\frac{2}{3}-\frac{1}{q}}\|\nabla\bar{\d}^{n-1}\|_{L_t^\infty(L^s)}\Big)\\
& \leq C
\left(\|\bar{\r}^n\|_{L_t^\infty(L^r)}+(t^{\frac{1}{2}-\frac{3}{2q}}+t^{\frac{1}{3}-\frac{1}{q}}+t^{\frac{2}{3}-\frac{1}{q}})\bar{U}^{n-1}(t)\right).
\end{split}
\end{equation}
Moreover, multiplying
$$\partial_t\bar{\r}^n+\u^n\cdot\nabla\bar{\r}^n=-\bar{\u}^{n-1}\cdot\nabla\r^n$$
 by $|\bar{\r}^n|^{r-2}\bar{\r}^n$ and integrating
over $\Omega$, using $\nabla\cdot\u^n=0$ and the zero
boundary condition, by H\"{o}lder's inequality, we have
\begin{equation*}
\begin{split}
\frac{1}{r}\frac{d}{dt}\|\bar{\r}^n\|_{L^r}^r&=-\frac{1}{r}\int_\Omega\u^n\cdot\nabla(|\bar{\r}^n|^r)
\ d\x-\int_\Omega|\bar{\r}^n|^{r-2}\bar{\r}^n\bar{\u}^{n-1}\cdot\nabla\bar{\r}^n\ d\x\\
&\leq
\|\bar{\r}^n\|_{L^r}^{r-1}\|\bar{\u}^{n-1}\cdot\nabla\bar{\r}^n\|_{L^r}.
\end{split}
\end{equation*}
By H\"{o}lder's inequality, \eqref{7} and the embedding
$$W^{2,s}(\Omega)\hookrightarrow L^\infty(\Omega) \quad \textrm{as}\ \  s>\frac{3}{2},$$
we eventually obtain
\begin{equation}\label{9}
\begin{split}
\|\bar{\r}^n(t)\|_{L^r}
&\le\int_0^t\|\bar{\u}^{n-1}(\tau)\cdot\nabla\r^n(\tau)\|_{L^r} d\tau\\
&\le
t^{1-\frac{1}{p}}\|\bar{\u}^{n-1}\|_{L^p_t(L^\infty)}\|\nabla\r^n\|_{L_t^\infty(L^r)}\\
&\le Ct^{1-\frac{1}{p}}\bar{U}^{n-1}(t).
\end{split}
\end{equation}
Inserting  \eqref{9} into  \eqref{8}, we get for $t\in [0,T_\ast]$,
\begin{equation*}
\bar{U}^n(t)\le
C(t^{1-\frac{1}{p}}+t^{\frac{1}{2}-\frac{3}{2q}}+t^{\frac{1}{3}-\frac{1}{q}}+t^{\frac{2}{3}-\frac{1}{q}})\bar{U}^{n-1}(t).
\end{equation*}
If we choose   $T_0\in (0,T_\ast]$ such that
\begin{equation}\label{10}
C(T_0^{1-\frac{1}{p}}+T_0^{\frac{1}{2}-\frac{3}{2q}}+T_0^{\frac{1}{3}-\frac{1}{q}}+T_0^{\frac{2}{3}-\frac{1}{q}})\leq
\frac{1}{2},
\end{equation}
then $\{(\r^n, \u^n, P^n, \d^n)\}$ is
a Cauchy sequence in $M_{T_0}^{p,s,r}$ and thus converges in
$M_{T_0}^{p,s,r}$.
\end{proof}

We remark here that the time of existence $T_0$ depends
(continuously) on the norms of the data, on the bound for the
density, on the domain and on the regularity parameters.

\subsection{The limit is a solution}
 Let $(\r,\u, P, \d)\in M_{{T_0}}^{p,s,r}$ be the limit of the sequence $\{(\r^n, \u^n, P^n, \d^n)\}_{n=1}^\infty$ in $M_{T_0}^{p,s,r}$.
Passing to the limit in \eqref{5088} and \eqref{7} yields
$$ \check{\r}\leq\r(\x,t)\leq\hat{\r},\ \ (\x,t)\in \O\times[0,T_0]\ \ \ \textrm{and}\ \ \r\in L^\infty(0,T_0; W^{1,r}(\Omega)),$$
$$\u\in L^\infty(0,T_0; D_{A_q}^{1-\frac{1}{p},p})\cap L^p(0,T_0; W^{2,q}(\Omega)), \quad \partial_t\u\in L^p(0,T_0; L^q(\Omega)),$$
$$\d\in L^\infty(0,T_0; B_{q,p}^{3(1-\frac{1}{p})})\cap L^p(0,T_0; W^{3,q}(\Omega)), \quad \partial_t\d\in L^p(0,T_0; L^q(\Omega)),$$
$$P\in L^p(0,T_0; W^{1,q}(\Omega)).$$
%Combining with the properties of convergence stated in the previous
%part of the proof, we conclude that $(\r^n, \u^n, P^n, \d^n)_{n\in
%\mathbb{N}}$ converges to $(\r, \u, P,\d)$ in $M_T^{p,q',r'}$ for
%all $q'<q$ and $r'<r$.

 We claim all those nonlinear terms in \eqref{r}
\eqref{d} \eqref{e5} converge to their corresponding terms in
\eqref{e2} almost everywhere in $\O\times (0,T_0)$.
 Indeed,  for $\alpha:=\frac{rs}{r+s}(=\frac{qr}{2q+r})$,
\begin{equation*}
\begin{split}
&\|\u^n\cdot\nabla\r^{n+1}-\u\cdot\nabla\r\|_{L^\infty_{T_0}(
L^\alpha)}\\
&\le\|\u^n-\u\|_{L^\infty_{T_0}(L^s)}\|\nabla\r^{n+1}\|_{L^\infty_{T_0}(
L^r)}+\|\u\|_{L^\infty_{T_0}
(L^s)}\|\nabla\r^{n+1}-\nabla\r\|_{L_{T_0}^\infty(L^r)}\\
&\le C\left(\varrho_0\|\u^n-\u\|_{M_{T_0}^{p,s,r}}
+\|\u\|_{L^\infty_{T_0}(
L^s)}\|\r^{n+1}-\r\|_{M_{T_0}^{p,s,r}}\right)\\&\rightarrow 0 \ \
\textrm{as}\ \ n\rightarrow\infty,
\end{split}
\end{equation*}

\begin{equation*}
\begin{split}
&\|\r^{n+1}\partial_t\u^{n+1}-\r\partial_t\u\|_{L^p_{T_0}(
L^\alpha)}\\
&\le\|\r^{n+1}\|_{L^\infty_{T_0}(L^r)}\|\partial_t\u^{n+1}-\partial_t\u\|_{L^p_{T_0}(
L^s)}+\|\r^{n+1}-\r\|_{L^\infty_{T_0}
(L^r)}\|\partial_t\u\|_{L_{T_0}^p(L^s)}\\
&\le C\left(\varrho_0\|\u^n-\u\|_{M_{T_0}^{p,s,r}}
+\|\partial_t\u\|_{L_{T_0}^p(L^s)}\|\r^{n+1}-\r\|_{M_{T_0}^{p,s,r}}\right)\\&\rightarrow
0 \ \ \textrm{as}\ \ n\rightarrow\infty,
\end{split}
\end{equation*}
and
\begin{equation*}
\begin{split}
&\|\r^{n+1}\u^n\cdot\nabla\u^n-\r\u\cdot\nabla\u\|_{L^p_{T_0}(L^\alpha)}\\
&\le\|\r^{n+1}-\r\|_{L^\infty_{T_0} (L^r)}\|\u^n\|_{L^\infty_{T_0}(
L^s)}\|\nabla\u^n\|_{L^p_{T_0} (L^\infty)}\\
&\quad+\|\r\|_{L^\infty_{T_0} (L^\infty)}\|\u^n\|_{L^\infty_T(L^s)}\|\nabla\u^n-\nabla\u\|_{L^p_T(L^r)}\\
&\quad+\|\r\|_{L^\infty_{T_0}
(L^\infty)}\|\u^n-\u\|_{L^\infty_{T_0}(
L^s)}\|\nabla\u\|_{L^p_{T_0}(L^r)}\\
&\le C\|\r^{n+1}-\r\|_{L^\infty_{T_0}
(L^r)}\|\u^n\|_{L^\infty_{T_0}(
L^q)}\|\u^n\|_{L^p_{T_0} (W^{2,q})}\\
&\quad+C\|\r\|_{L^\infty_{T_0} (L^\infty)}\|\u^n\|_{L^\infty_{T_0}(
L^q)}\|\nabla\u^n-\nabla\u\|_{L^p_{T_0}(L^r)}\\
&\quad+\|\r\|_{L^\infty_{T_0}(L^\infty)}\|\u^n-\u\|_{L^\infty_{T_0}(L^s)}\|\nabla\u\|_{L^p_{T_0}(L^r)}\\
&\le C\Big((U^0)^2\|\r^{n+1}-\r\|_{M_{T_0}^{p,s,r}}
+U^0{T_0}^{\frac{1}{2}-\frac{3}{2q}}\|\r\|_{L^\infty_{T_0}
(L^\infty)}\|\u^n-\u\|_{M_{T_0}^{p,s,r}}\\
&\quad\qquad+{T_0}^{\frac{1}{2}-\frac{3}{2q}}\|\r\|_{L^\infty_{T_0}
(L^\infty)}\|\u^n-\u\|_{M_{T_0}^{p,s,r}}
\|\u\|_{M_{T_0}^{p,s,r}}\Big)\\&\rightarrow 0 \ \ \textrm{as}\ \
n\rightarrow\infty,
\end{split}
\end{equation*}
due to $\u^n\rightarrow\u$ and $\r^{n+1}\rightarrow\r$ in
$M_{T_0}^{p,s,r}$ as $n\rightarrow \infty$. Hence,
$$\u^n\cdot\nabla\r^{n+1}\rightarrow\u\cdot\nabla\r \quad \textrm{in}\ \  \big(L^p(0,{T_0};
L^\alpha(\Omega))\big)^3;$$
$$\r^{n+1}\partial_t\u^{n+1}\rightarrow\r\partial_t\u \quad \textrm{in}\ \ \big(L^p(0,{T_0};
L^\alpha(\Omega))\big)^3;$$
$$\r^{n+1}\u^n\cdot\nabla\u^n\rightarrow\r\u\cdot\nabla\u \quad \textrm{in}\ \ \big(L^p(0,{T_0};
L^\alpha(\Omega))\big)^3.$$ Meanwhile,
\begin{equation*}
\begin{split}
&\|\u^n\cdot\nabla\d^n-\u\cdot\nabla\d\|_{L^p_{T_0}(
L^s)}\\
&\le\|\d^n-\d\|_{L^\infty_{T_0}(L^s)}\|\nabla\d^n\|_{L^p_{T_0}(
L^\infty)}+\|\u\|_{L^\infty_{T_0}(
L^s)}\|\nabla\d^n-\nabla\d\|_{L^p_{T_0}( L^\infty)}\\
&\le C\left(U^0\|\d^n-\d\|_{M_{{T_0}}^{p,s,r}}
+\|\u\|_{L^\infty_{T_0}(
L^s)}\|\d^n-\d\|_{M_{{T_0}}^{p,s,r}}\right)\\
&\rightarrow 0,\ \textrm{as}\ n\rightarrow\infty,
\end{split}
\end{equation*}
\begin{equation*}
\begin{split}
&\|\nabla\cdot((\nabla\d^n)^\top\nabla\d^n)-\nabla\cdot((\nabla\d)^\top\nabla\d)\|_{L^p_{T_0}(
L^\alpha)}\\
&=\|\frac{1}{2}\nabla\left((\nabla\d^n+\nabla\d):(\nabla\d^n-\nabla\d)\right)
+(\nabla\d^n-\nabla\d)^\top\triangle\d^n+(\nabla\d)^\top(\triangle\d^n-\triangle\d)\|_{L^p_{T_0}(
L^\alpha)}\\
&\le
C\Big(\|\nabla\d^n-\nabla\d\|_{L^\infty_{T_0}(L^s)}\|\triangle\d^n+\triangle\d\|_{L^p_{T_0}(L^r)}
+\|\nabla\d^n+\nabla\d\|_{L^\infty_{T_0}(L^s)}\|\triangle\d^n-\triangle\d\|_{L^p_{T_0}(L^r)}\\
&
\qquad\quad+\|\nabla\d^n-\nabla\d\|_{L^\infty_{T_0}(L^s)}\|\triangle\d^n\|_{L^p_{T_0}(L^r)}+\|\nabla\d\|_{L^\infty_{T_0}(L^s)}\|\triangle\d^n-\triangle\d\|_{L^p_{T_0}(L^r)}\Big)\\
&\le
CT_0^{\frac{1}{3}-\frac{1}{q}}(U^0+\|\d\|_{M_{{T_0}}^{p,s,r}})\|\d^n-\d\|_{M_{{T_0}}^{p,s,r}}
\\
&\rightarrow 0,\ \textrm{as}\ n\rightarrow\infty.
\end{split}
\end{equation*}
Then, we have
$$\nabla\u^n\cdot\nabla\d^n\rightarrow \u\cdot\nabla\d
\quad\textrm{in}\ \big(L^p(0,T_0; L^s(\Omega))\big)^3;$$
$$\nabla\cdot((\nabla\d^n)^\top\nabla\d^n)\rightarrow\nabla\cdot((\nabla\d)^\top\nabla\d) \quad\textrm{in}\
\big(L^p(0,T_0; L^\alpha(\Omega))\big)^3.$$ If $s>3$, we have
\begin{equation*}
\begin{split}
&\||\nabla\d^n|^2\d^n- |\nabla\d|^2\d\|_{L^p_{T_0}(L^{\frac{q}{3}})}\\
&\leq \|\nabla\d^n\|_{L^\infty_{T_0}(L^q)}
\left(\|\nabla\d^n\|_{L^p_{T_0}(L^\infty)}\|\d^n-\d\|_{L^\infty_{T_0}(L^{\frac{q}{2}})}
+\|\d\|_{L^p_{T_0}(L^\infty)}\|\nabla\d^n-\nabla\d\|_{L^\infty_{T_0}(L^{\frac{q}{2}})}\right)\\
&\quad
+\|\d\|_{L^\infty_{T_0}(L^\infty)}\|\nabla\d\|_{L^p_{T_0}(L^q)}\|\nabla\d^n-\nabla\d\|_{L^\infty_{T_0}(L^{\frac{q}{2}})}\\
&\leq C\|\nabla\d^n\|_{L^\infty_{T_0}(L^q)}
\left(\|\nabla\d^n\|_{L^p_{T_0}(W^{1,q})}\|\d^n-\d\|_{L^\infty_{T_0}(L^s)}
+\|\d\|_{L^p_{T_0}(W^{1,s}))}\|\nabla\d^n-\nabla\d\|_{L^\infty_{T_0}(L^s)}\right)\\
&\quad
+C\|\d\|_{L^\infty_{T_0}(W^{1,s}))}\|\nabla\d\|_{L^p_{T_0}(L^r)}\|\nabla\d^n-\nabla\d\|_{L^\infty_{T_0}(L^s)}\\
&\leq CU^0\left(U^0+\|\d\|_{M_{{T_0}}^{p,s,r}}\right)\|\d^n-\d\|_{M_{{T_0}}^{p,s,r}}+CT_0^{\frac{1}{2}-\frac{3}{2q}}\|\d\|^2_{M_{{T_0}}^{p,s,r}}\|\d^n-\d\|_{M_{{T_0}}^{p,s,r}}\\
&\rightarrow 0,\ \textrm{as}\ n\rightarrow\infty.
\end{split}
\end{equation*}
If $\frac{3}{2}<s<3$, we have
\begin{equation*}
\begin{split}
&\||\nabla\d^n|^2\d^n- |\nabla\d|^2\d\|_{L^p_{T_0}(L^\frac{q}{3})}\\
&\leq C\left(\||\nabla\d^n|^2
(\d^n-\d)\|_{L^p_{T_0}(L^{\frac{q}{3}})}
+\|\nabla\d^n:(\nabla\d^n-\nabla\d)\d\|_{L^p_{T_0}(L^{\frac{q}{3}})}\right)\\
&\quad
+\|\nabla\d:(\nabla\d^n-\nabla\d)\d\|_{L^p_{T_0}(L^\frac{3s}{6-s})}\\
&\leq C\|\nabla\d^n\|_{L^\infty_{T_0}(L^q)}
\left(\|\nabla\d^n\|_{L^p_{T_0}(L^\infty)}\|\d^n-\d\|_{L^\infty_{T_0}(L^{\frac{q}{2}})}
+\|\d\|_{L^p_{T_0}(L^\infty)}\|\nabla\d^n-\nabla\d\|_{L^\infty_{T_0}(L^{\frac{q}{2}})}\right)\\
&\quad
+\|\d\|_{L^\infty_{T_0}(L^\frac{3s}{3-s})}\|\nabla\d\|_{L^p_{T_0}(L^\infty)}\|\nabla\d^n-\nabla\d\|_{L^\infty_{T_0}(L^s)}\\
&\leq C\|\nabla\d^n\|_{L^\infty_{T_0}(L^q)}
\left(\|\nabla\d^n\|_{L^p_{T_0}(W^{1,q})}\|\d^n-\d\|_{L^\infty_{T_0}(L^s)}
+\|\d\|_{L^p_{T_0}(W^{2,s}))}\|\nabla\d^n-\nabla\d\|_{L^\infty_{T_0}(L^s)}\right)\\
&\quad
+C\|\d\|_{L^\infty_{T_0}(W^{1,s}))}\|\nabla\d\|_{L^p_{T_0}(W^{2,s})}\|\nabla\d^n-\nabla\d\|_{L^\infty_{T_0}(L^s)}\\
&\leq CU^0\left(U^0+\|\d\|_{M_{{T_0}}^{p,s,r}}\right)\|\d^n-\d\|_{M_{{T_0}}^{p,s,r}}+C\|\d\|^2_{M_{{T_0}}^{p,s,r}}\|\d^n-\d\|_{M_{{T_0}}^{p,s,r}}\\
&\rightarrow 0,\ \textrm{as}\ n\rightarrow\infty.
\end{split}
\end{equation*}
The  case $s=3$ may be handled by noticing that we also have
\begin{equation*}
\begin{split}
&\||\nabla\d^n|^2\d^n- |\nabla\d|^2\d\|_{L^p_{T_0}(L^{\frac{q}{3}})}\\
&\leq \|\nabla\d^n\|_{L^\infty_{T_0}(L^q)}
\left(\|\nabla\d^n\|_{L^p_{T_0}(L^\infty)}\|\d^n-\d\|_{L^\infty_{T_0}(L^{\frac{q}{2}})}
+\|\d\|_{L^p_{T_0}(L^\infty)}\|\nabla\d^n-\nabla\d\|_{L^\infty_{T_0}(L^{\frac{q}{2}})}\right)\\
&\quad
+\|\d\|_{L^\infty_{T_0}(L^\frac{3q}{9-q})}\|\nabla\d\|_{L^p_{T_0}(L^\infty)}\|\nabla\d^n-\nabla\d\|_{L^\infty_{T_0}(L^3)}\\
&\leq C\|\nabla\d^n\|_{L^\infty_{T_0}(L^q)}
\left(\|\nabla\d^n\|_{L^p_{T_0}(W^{1,q})}\|\d^n-\d\|_{L^\infty_{T_0}(L^3)}
+\|\d\|_{L^p_{T_0}(W^{2,3}))}\|\nabla\d^n-\nabla\d\|_{L^\infty_{T_0}(L^3)}\right)\\
&\quad
+C\|\d\|_{L^\infty_{T_0}(W^{1,3})}\|\nabla\d\|_{L^p_{T_0}(W^{2,3})}\|\nabla\d^n-\nabla\d\|_{L^\infty_{T_0}(L^3)}\\
&\leq CU^0\left(U^0+\|\d\|_{M_{{T_0}}^{p,3,r}}\right)\|\d^n-\d\|_{M_{{T_0}}^{p,3,r}}+C\|\d\|^2_{M_{{T_0}}^{p,3,r}}\|\d^n-\d\|_{M_{{T_0}}^{p,3,r}}\\
&\rightarrow 0,\ \textrm{as}\ n\rightarrow\infty.
\end{split}
\end{equation*}
 Hence, we finally get
$$|\nabla\d^n|^2\d^n\rightarrow |\nabla\d|^2\d\quad\textrm{in}\
\big(L^p(0,{T_0}; L^{\frac{q}{3}}(\O))\big)^3.$$
 Thus, passing to the limit in \eqref{r}, \eqref{d} and \eqref{e5} as
$n\to\infty$ , since $L^s(\O)\hookrightarrow
L^\alpha(\O)\hookrightarrow L^{\frac{q}{3}}(\O)$, we conclude that
\eqref{e2} holds in $\big(L^p(0,T_0;
L^{\frac{q}{3}}(\Omega))\big)^3$ and therefore almost everywhere in
$\O\times (0,T_0)$.

Multiply the $\d$-system \eqref{e23} by $\d$, we obtain
\begin{equation*}
\frac{1}{2}\partial_t(|\d|^2)
+\frac{1}{2}\u\cdot\nabla(|\d|^2)=\triangle\d\cdot\d+|\nabla\d|^2|\d|^2.
\end{equation*}
Since $$\triangle(|\d|^2)=2|\nabla\d|^2+2\d\cdot(\triangle\d),$$
then it follows that
\begin{equation*}
\frac{1}{2}\partial_t(|\d|^2)+\frac{1}{2}\u\cdot\nabla(|\d|^2)=\frac{1}{2}\triangle(|\d|^2)-|\nabla\d|^2+|\nabla\d|^2|\d|^2.
\end{equation*}
Therefore, it is easy to deduce that
\begin{equation}\label{s}
\partial_t(|\d|^2-1)
-\triangle(|\d|^2-1)+\u\cdot\nabla(|\d|^2-1)-2|\nabla\d|^2(|\d|^2-1)=0.
\end{equation}
Multiplying \eqref{s} by $(|\d|^2-1)$ and then integrating over
$\Omega$, using \eqref{e24} and \eqref{bc}, we get the following
inequality:
\begin{equation}\label{ss}
\begin{split}
\frac{d}{dt}\int_{\Omega} (|\d|^2-1)^2\ d\x &
\le 4\int_{\Omega}|\nabla\d|^2(|\d|^2-1)^2\ d\x\\
 &\le 4\|\nabla\d\|_{L^\infty}^2\int_{\Omega}(|\d|^2-1)^2\ d\x.
\end{split}
\end{equation}
Remark that interpolation between $L^\infty(0,T_0; W^{1,q}(\O))$ and
$L^p(0,T_0;W^{3,q}(\O))$  shows that for some  positive
$\alpha>\frac{1}{2}$,
 $\d$ belongs to $L^2(0,{T_0}; H^{2+\alpha}(\O))$
and that $\|\nabla\d\|_{L^\infty}^2\in L^1(0,T_0)$. Notice that
$$\int_{\Omega}(|\d|^2-1)^2\ d\x=0, \quad\text{at time}\ t=0.$$
Thus, using \eqref{ss} together with Gr\"{o}nwall's inequality, it
yields $|\d|=1$
 in $\O\times (0, T_0)$.

\subsection{Uniqueness and continuity}
Let $(\r_1, \u_1,  P_1, \F_1)$ and $(\r_2, \u_2,
P_2, \F_2)$ be two solutions to \eqref{e2} with the initial-boundary
conditions \eqref{ic} \eqref{bc}. Denote
$$\bar{\r}=\r_1-\r_2,\quad\bar{\u}=\u_1-\u_2,\quad \bar{P}=P_1-P_2,\quad \bar{\F}=\F_1-\F_2.$$ Note
that the quadruplet $(\bar{\r}, \bar{\u}, \bar{P}, \bar{\F})$
satisfies the following system:
\begin{equation*}
\begin{cases}
\partial_t\bar{\r}+\u_1\cdot\nabla\bar{\r}=-\bar{\u}\cdot\nabla\r_2,\\
\r_1\partial_t\bar{\u}-\triangle\bar{\u}+\nabla
\bar{P}=-\bar{\r}(\partial_t\u_2+\u_1\cdot\nabla\u_1)-\r_2(\bar{\u}\cdot\nabla\u_1+\u_2\cdot\nabla\bar{\u})\\
\qquad\qquad\qquad\qquad\quad\ -\Dv\big((\nabla\d_1)^\top\nabla\bar{\d}\big)-\Dv\big((\nabla\bar{\d})^\top\nabla\d_2\big),\\
\partial_t\bar{\F}-\D\bar{\F}=-\u_1\cdot\nabla\bar{\d}-\bar{\u}\cdot\nabla\F_2+|\nabla\d_1|^2\bar{\d}+\left((\nabla\d_1+\nabla\d_2):\nabla\bar{\d}\right)\d_2,\\
\Dv\bar{\u}=0, \quad \int_\O \bar{P}\ d\x=0
\end{cases}
\end{equation*}
with the initial-boundary conditions:
$$(\bar{\r},\bar{\u},\bar{\d})|_{t=0}=(0,0,0), \quad (\bar{\u},\partial_{\bf \nu}\bar{\d})|_{\partial\O}=(0,0).$$
Using the same argument for $\bar{\r}^n$ in Subsection 4.3, for all
$t\in [0, T_0]$, we have
\begin{equation}\label{151}
\begin{split}
\|\bar{\r}(t)\|_{L^r}&\leq
\int_0^t\|\nabla\r_2(\tau)\|_{L^r}\|\bar{\u}(\tau)\|_{L^\infty} d\tau\\
&\leq
t^{1-\frac{1}{p}}\|\nabla\r_2\|_{L_t^\infty(L^r)}\|\bar{\u}\|_{L_t^p(L^\infty)}\\
&\leq
Ct^{1-\frac{1}{p}}\|\r_2\|_{L_t^\infty(W^{1,r})}\|\bar{\u}\|_{L_t^p(W^{2,s})}.
\end{split}
\end{equation}

On the one hand, since $\r_1, \r_2\in L^\infty(0,T_0;
W^{1,r}(\Omega))\cap W^{1,\infty}(0,T_0; L^s(\Omega))$ implies that
$\r_1, \r_2\in C^\beta([0,T_0]; L^\infty(\Omega))$ whenever $\beta
\in (0, \frac{1-\frac{3}{r}}{1+\frac{3}{q}})$, then Theorem \ref{T4}
yields, for some constant C depending on $T_0, p, q, r, \check{\r},
\hat{\r}, \Omega, \beta$ and on the norm of $\r_1$ in
$L^\infty(0,T_0; W^{1,r}(\O))\cap C^\beta(0,T_0; L^\infty(\O))$, and
for all $t\in [0,T_0]$,
\begin{equation}\label{511}
\begin{split}
&\|\bar{\u}(t)\|_{D_{A_s}^{1-\frac{1}{p},
p}}+\|\bar{\u}\|_{L_t^p(W^{2,s})}+\|\partial_t\bar{\u}\|_{L_t^p(L^s)}+\|\bar{P}\|_{L_t^p(W^{1,s})}\\
&\leq
C\Big(\|\bar{\r}(\partial_t\u_2+\u_1\cdot\nabla\u_1)\|_{L_t^p(L^s)}+\|\r_2\bar{\u}\cdot\nabla\u_1\|_{L_t^p(L^s)}+\|\r_2\u_2\cdot\nabla\bar{\u}\|_{L_t^p(L^s)}\\
&\qquad\quad+\|\nabla\cdot\big((\nabla\d_1)^\top
\nabla\bar{\d}\big)\|_{L_t^p(L^s)}+\
\|\nabla\cdot\big((\nabla\bar{\d})^\top\nabla\d_2\big)\|_{L_t^p(L^s)}\Big)\\
&\leq
C\Big(\|\bar{\r}\|_{L_t^\infty(L^r)}\big(\|\partial_t\u_2\|_{L_t^p(L^q)}+\|\u_1\|_{L_t^\infty(L^q)}\|\nabla\u_1\|_{L_t^p(L^\infty)}\big)+\|\nabla\u_1\|_{L_t^p(L^\infty)}\|\bar{\u}\|_{L_t^\infty(L^s)}\\
&\qquad\quad
 +\|\u_2\|_{L_t^\infty(L^q)}\|\nabla\bar{\u}\|_{L_t^p(L^r)}+\|\nabla\d_1\|_{L_t^\infty(L^q)}\|\triangle\bar{\d}\|_{L_t^p(L^r)}
 +\|\triangle\d_1\|_{L_t^p(L^\infty)}\|\nabla\bar{\d}\|_{L_t^\infty(L^s)}
\\
&\qquad\quad+\|\nabla\d_2\|_{L_t^\infty(L^q)}\|\triangle\bar{\d}\|_{L_t^p(L^r)}+
\|\triangle\d_2\|_{L_t^p(L^\infty)}\|\nabla\bar{\d}\|_{L_t^\infty(L^s)}\Big)\\
&\leq
C\Big(\|\bar{\r}\|_{L_t^\infty(L^r)}\big(\|\partial_t\u_2\|_{L_t^p(L^q)}+\|\u_1\|_{L_t^\infty({D_{A_q}^{1-\frac{1}{p},p}})}\|\u_1\|_{L_t^p(W^{2,q})}\big)+\|\u_1\|_{L_t^p(W^{2,q})}\|\bar{\u}\|_{L_t^\infty(L^s)}\\
&\qquad\quad
+\|\u_2\|_{L_t^\infty({D_{A_q}^{1-\frac{1}{p},p}})}\|\nabla\bar{\u}\|_{L_t^p(L^r)}+(\|\d_1\|_{L_t^\infty(B_{q,p}^{3(1-\frac{1}{p})})}+\|\d_2\|_{L_t^\infty(B_{q,p}^{3(1-\frac{1}{p})})})\|\triangle\bar{\d}\|_{L_t^p(L^r)}
\\
&\qquad\quad+(\|\d_1\|_{L_t^p(W^{3,q})}+
\|\d_2\|_{L_t^p(W^{3,q})})\|\nabla\bar{\d}\|_{L_t^\infty(L^s)}\Big).
\end{split}
\end{equation}
 On the other hand, Theorem
\ref{T3} yields,  for some constant C independent of $T_0$,
\begin{equation}\label{512}
\begin{split}
&\|\bar{\d}(t)\|_{B_{s,p}^{3(1-\f{1}{p})}}+\|\bar{\d}\|_{\mathcal{W}_{s,p}(0,t)}\\
&\le
C\left(\|-\u_1\cdot\nabla\bar{\d}-\bar{\u}\cdot\nabla\d_2+|\nabla\d_1|^2\bar{\d}+\big((\nabla\d_1+\nabla\d_2):\nabla\bar{\d}\big)\d_2\|_{L^p_t
(L^s)}\right)\\
&\le
C\Big(\|\u_1\|_{L_t^\infty(L^q)}\|\nabla\bar{\d}\|_{L_t^p(L^r)}+\|\nabla\d_2\|_{L_t^p(L^\infty)}\|\bar{\u}\|_{L_t^\infty(L^s)}\\
&\qquad\quad+\|\nabla\d_1\|_{L_t^\infty(L^q)}\|\nabla\d_1\|_{L_t^p(L^\infty)}\|\bar{\d}\|_{L_t^\infty(L^r)}\\
&\qquad\quad+(\|\nabla\d_1\|_{L_t^p(L^\infty)}+\|\nabla\d_2\|_{L_t^p(L^\infty)})\|\nabla\bar{\d}\|_{L_t^\infty(L^s)}\|\d_2\|_{L_t^\infty(L^\infty)}\Big)\\
&\le
C\Big(\|\u_1\|_{L_t^\infty({D_{A_q}^{1-\frac{1}{p},p}})}\|\nabla\bar{\d}\|_{L_t^p(L^r)}+\|\d_2\|_{L_t^p(W^{2,q})}\|\bar{\u}\|_{L_t^\infty(L^s)}\\
&\qquad\quad+\|\d_1\|_{L_t^\infty(B_{q,p}^{3(1-\frac{1}{p})})}\|\d_1\|_{L_t^p(W^{2,q})}\|\bar{\d}\|_{L_t^\infty(L^r)}\\
&\qquad\quad+\|\d_2\|_{L_t^\infty(B_{q,p}^{3(1-\frac{1}{p})})}(\|\d_1\|_{L_t^p(W^{2,q})}+\|\d_2\|_{L_t^p(W^{2,q})})\|\nabla\bar{\d}\|_{L_t^\infty(L^s)}\Big).
\end{split}
\end{equation}
We remark here that H\"{o}lder's inequality and  the embedding
$W^{1,q}(\O)\hookrightarrow L^\infty(\O) \ (q>3)$ have been employed
repeatedly in both \eqref{511} and \eqref{512}.

Lemmas \ref{l1}-\ref{l3} yield, by use of Young's inequality,
$$\|\nabla\bar{\u}\|_{L_t^p(L^r)}\leq C
t^{\frac{1}{2}-\frac{3}{2q}}\big(\|\bar{\u}\|_{L_t^\infty(D_{A_s}^{1-\frac{1}{p},p})}+\|\bar{\u}\|_{L_t^p(W^{2,s})}\big),$$
$$\|\nabla\bar{\d}\|_{L_t^p(L^r)}\leq C
t^{\frac{2}{3}-\frac{1}{q}}\big(\|\bar{\d}\|_{L_t^\infty(B_{s,p}^{3(1-\frac{1}{p})})}+\|\bar{\d}\|_{L_t^p(W^{3,s})}\big),$$
$$\|\triangle\bar{\d}\|_{L_t^p(L^r)}\leq C
t^{\frac{1}{3}-\frac{1}{q}}\big(\|\bar{\d}\|_{L_t^\infty(B_{s,p}^{3(1-\frac{1}{p})})}+\|\bar{\d}\|_{L_t^p(W^{3,s})}\big).$$

 Define
\begin{equation*}
\begin{split}
X(t):=&\|\bar{\r}\|_{L_t^\infty(L^r)}+\|\bar{\u}\|_{L^\infty_t(
D_{A_s}^{1-\f{1}{p},p})}+\|\bar{\u}\|_{L^p_t(
W^{2,s})}+\|\partial_t\bar{\u}\|_{L^p_t(L^s)}\\
&+\|\nabla\bar{ P}\|_{L^p_t( L^s)}+\|\bar{\F}\|_{L^\infty_t(
B_{s,p}^{3(1-\f{1}{p})})}+\|\bar{\F}\|_{\mathcal{W}_{s,p}(0,t)}.
\end{split}
\end{equation*}
Thus, combining \eqref{151}-\eqref{512} and
$B_{s,p}^{3(1-\f{1}{p})}\hookrightarrow W^{1,s}(\O)\hookrightarrow L^r(\O)$, we have
\begin{equation*}
\begin{split}
X(t)&\le C\Big\{t^{1-\frac{1}{p}}\|\r_2\|_{L^\infty_t(W^{1,r})}\big(1+\|\partial_t\u_2\|_{L^p_t(L^q)}+\|\u_1\|_{L^\infty_t(D_{A_q}^{1-\frac{1}{p},p})}\|\u_1\|_{L^p_t(W^{2,q})}\big)\\
&\qquad\quad +t^{\frac{1}{2}-\frac{3}{2q}}\|\u_2\|_{L_t^\infty(D_{A_q}^{1-\frac{1}{p},p})}+t^{\frac{2}{3}-\frac{1}{q}}\|\u_1\|_{L_t^\infty(
D_{A_q}^{1-\frac{1}{p},p})}+\|\u_1\|_{L_t^p(W^{2,q})}\\
&\qquad\quad+t^{\frac{1}{3}-\frac{1}{q}}\big(\|\d_1\|_{L_t^\infty(B_{q,p}^{3(1-\frac{1}{p})})}+\|\d_2\|_{L_t^\infty(B_{q,p}^{3(1-\frac{1}{p})})}\big)+\|\d_2\|_{L_t^p(W^{2,q})}\\
&\qquad\quad+\|\d_1\|_{L_t^p(W^{3,q})}+
\|\d_2\|_{L_t^p(W^{3,q})}+\|\d_1\|_{L_t^\infty(B_{q,p}^{3(1-\frac{1}{p})})}\|\d_1\|_{L_t^p(W^{2,q})}\\
&\qquad\quad+\|\d_2\|_{L_t^\infty(B_{q,p}^{3(1-\frac{1}{p})})}\big(\|\d_1\|_{L_t^p(W^{2,q})}+\|\d_2\|_{L_t^p(W^{2,q})}\big)\Big\}X(t).
\end{split}
\end{equation*}
Now, choosing $\eta$ so small that the term between brackets is less
than $\frac{1}{2}$ for $t=\eta$ enables us to conclude that $X\equiv
0$ on $[0,\eta]$. As the constant $C$ does not depend on $\eta$, a
standard induction argument yields the uniqueness on $[0,T_0]$.

Finally, as $\r$ satisfies a transport equation with data in
$W^{1,r}(\Omega)$,\ $\u$ satisfies
$$\r\partial_t\u-\triangle\u+\nabla P \in \left(L^p(0,t; L^q(\Omega))\right)^3,$$
and $\d$ satisfies $$\partial_t\d-\triangle\d\in \left(L^p(0,t;
L^q(\Omega))\right)^3,$$ then, Proposition \ref{P1}, Theorems
\ref{T3}-\ref{T4} insure that $\r\in C([0,T_0]; W^{1,r}(\Omega))$
(if $r\neq \infty$), $\u\in C([0,T_0]; D_{A_q}^{1-\f{1}{p},p})$ and
$\d\in C([0,T_0]; B_{q,p}^{3(1-\f{1}{p})})$.

\begin{Remark}
Following the argument of uniqueness and continuity, we can also
easily prove that if $(\r_1, \u_1, P_1, \d_1)$ and $(\r_2, \u_2,
P_2, \d_2)$ are solutions to \eqref{e2}-\eqref{bc} with different
initial data $(\r_0^1, \u_0^1, \d_0^1)$ and $(\r_0^2, \u_0^2,
\d_0^2)$, then the following estimate holds true on $[0,T_0]$:
\begin{equation*}
\begin{split}
&\|\bar{\r}(t)\|_{L^r}+\|\bar{\u}(t)\|_{D_{A_s}^{1-\frac{1}{p},p}}+\|\bar{\u}\|_{L_t^p(W^{2,s})}+\|\partial_t\bar{\u}\|_{L_t^p(L^s)}\\&
+\|\bar{P}\|_{L_t^p(W^{1,s})}+\|\bar{\d}(t)\|_{B_{s,p}^{3(1-\f{1}{p})}}+\|\bar{\d}\|_{\mathcal{W}_{s,p}(0,t)}\\
&\leq
C\left(\|\bar{\r}_0\|_{L^r}+\|\bar{\u}_0\|_{D_{A_s}^{1-\frac{1}{p},p}}+\|\bar{\d}_0\|_{B_{s,p}^{3(1-\f{1}{p})}}\right),
\end{split}
\end{equation*}
where $\bar{\r}_0:=\r_0^1-\r_0^2,\  \bar{\u}_0:=\u_0^1-\u_0^2, \
\bar{\d}_0:=\d_0^1-\d_0^2$. Combining with Theorem \ref{T1}, we
conclude that for small enough $T$, the map $(\r_0, \u_0,
\d_0)\rightarrow(\r, \u, P,\d )$ is Lipschitz continuous from
bounded sets of $W^{1,r}\times D_{A_q}^{1-\frac{1}{p}, p}\times
B_{q,p}^{3(1-\f{1}{p})}$ to
\begin{equation*}
\begin{split}
&C([0,T]; L^r(\O))\times \Big(C([0,T];
D_{A_s}^{1-\frac{1}{p},p})\cap \big(W^{1,p}(0,T; L^s(\O))\big)^3\cap
\big(L^p(0,T;
W^{2,s}(\O))\big)^3\Big)\\
&\times L^p(0,T; W^{1,s}(\O))\times \Big(C([0,T];
B_{s,p}^{3(1-\f{1}{p})})\cap \big(W^{1,p}(0,T; L^s(\O))\big)^3\cap
\big(L^p(0,T; W^{2,s}(\O))\big)^3\Big).
\end{split}
\end{equation*}
\end{Remark}

\bigskip

%%%%%%%%%%%%%%%%%

\section{Global Existence}

In this section, we prove that, if  the initial data of velocity and
orientation field  is sufficiently small in appropriate norms, the local strong solution
$(\r, \u, P, \d)$ of \eqref{e2}-\eqref{bc} established in the
previous section is indeed global in time.

\subsection{Estimates for $\|\u\|_{L^2}$ and $\|\d\|_{L^2}$}

\begin{Lemma}\label{51}
Let $\O, p, q, r$ be as in Theorem \ref{T1} and let $(\r, \u, P,
\d)\in M_{T_0}^{p, q, r}$ be a solution to \eqref{e2}-\eqref{bc} on
$\O\times [0,T_0]$. Then the following inequality holds true for all
\ $t\in[0,T_0]:$
\begin{equation*}%\label{513}
\begin{split}
&\|(\sqrt{\r}\u)(t)\|_{L^2}^2+\|\nabla\d(t)\|_{L^2}^2\\&\leq
e^{-\frac{2\lambda_1}{\hat{\r}}
t}\left(\|\sqrt{\r_0}\u_0\|_{L^2}^2+\|\nabla\d_0\|_{L^2}^2\right)\left(1+
\frac{2\lambda_1}{\hat{\r}} t e^{
\frac{2\lambda_1}{\hat{\r}}t}\right),
\end{split}
\end{equation*}
where $\lambda_1$ stands for the first eigenvalue of the
Dirichlet-Laplace operator in $\Omega$.
\end{Lemma}

\begin{proof}
Due to the inhomogeneous incompressible character the flows we are
dealing with, the natural framework in which we shall work is that
of the solenoidal vector field of $L^2(\O)^3$. Note that
$$\u\in C([0,{T_0}]; D_{A_q}^{1-\f{1}{p},p})\cap\left(
L^p(0,{T_0}; W^{2,q}(\O)\cap W_0^{1,q}(\O))\right)^3,$$
$$\d\in C([0,T_0]; B_{q,p}^{3(1-\f{1}{p})})\cap \left(L^p(0,T_0;W^{3,q}(\O))\right)^3.$$  And, since $$D_{A_q}^{1-\f{1}{p},p}\hookrightarrow B^{2(1-\f{1}{p})}_{q,p}\cap
X^q,$$ (see Proposition 2.5 in \cite{D}) where $$X^q=\{{\bf z}\in
L^q(\O)^3\mid \nabla\cdot {\bf z}=0\ \ \textrm{in}\ \Omega \ \
\textrm{and} \ {\bf z}\cdot{\bf n}=0\ \ \textrm{on}\
\partial\Omega\},$$
% Sobolev's embedding $W^{1,q}(\Omega)\hookrightarrow H^1(\Omega)$ as $q>3$,
then, when $1<p<2,$ by the standard interpolation inequality
 $$L^\infty(0,T_0; L^q(\Omega))\cap L^p(0,T_0;W^{2,q}(\Omega))\subset L^2(0,T_0;H^{1+\alpha}(\Omega)),$$
 where $$\frac{1}{2}=\frac{1-\theta}{\infty}+\frac{\theta}{p}=\frac{\theta}{p},
 \quad \frac{1}{2}-\frac{1+\alpha}{3}=(1-\theta)\frac{1}{q}+\theta\left(\frac{1}{q}-\frac{2}{3}\right), $$
we have
\begin{equation}\label{388}
\u\in C([0,T_0]; H^\alpha(\O))\cap \left(L^2(0,T_0; H^{1+\a}(\O))\right)^3,
\end{equation}
\begin{equation}\label{3881}
\d\in C([0,T_0]; H^{1+\alpha}(\O))\cap \left(L^2(0,T_0; H^{2+\a}(\O))\right)^3.
\end{equation}
When $2\leq p<\infty$, \eqref{388}-\eqref{3881} hold obviously due to
$W^{2,q}(\O)\hookrightarrow H^{2}(\O) \ \textrm{as} \  q>3.$

Now, $\r$ is continuous in $(t, \x)$, $\u\in C([0,T_0];
H^\alpha(\O))\cap \left(L^2(0,T_0; H^{1+\a}(\O))\right)^3$ and  $\d\in C([0,T_0];
H^{1+\alpha}(\O))\cap \left(L^2(0,T_0; H^{2+\a}(\O))\right)^3$. This enables us to
justify the following computations.

Taking the $L^2$ scalar product %\big(i.e.,\ $(\phi,\varphi)=\sum_{i=1}^{3}\int_\Omega\phi_i\varphi_i\ dx$\big)
in \eqref{e22} with $\u$ and performing integration by parts, using
the continuity equation \eqref{e21},  we obtain
\begin{equation}\label{514}
\begin{split}
&\frac{1}{2}\f{d}{dt}\int_\O\r|\u|^2\ d\x+\int_\O|\nabla\u|^2\ d\x\\
%\\&=\int_\O|\u|^2\partial_t\r\ d\x-\int_\O\r\u\cdot(\nabla\u)\cdot\u\ d\x-\int_\O \H\cdot(\nabla\u)\H\ d\x\\
&=-\frac{1}{2}\int_\O|\u|^2\nabla\cdot(\r\u)\
d\x-\int_\O\r\u\cdot\nabla\u\cdot\u\
d\x-\int_\O\u\cdot\big((\nabla\d)^\top\triangle\d\big)\ d\x\\
&=-\int_\O\u\cdot\big((\nabla\d)^\top\triangle\d\big)\ d\x.
\end{split}
\end{equation}
 Here we have used the facts
$$\nabla\cdot(\nabla\d\odot\nabla\d)=\nabla\left(\frac{|\nabla\d|^2}{2}\right)+(\nabla\d)^\top\triangle\d,$$
and  $\nabla\cdot\u=0$ in $\Omega$, $\u=0$ on $\partial\Omega$, as
well as
$$\int_\O\nabla P\cdot\u \ d\x=\int_\O \nabla\left(\frac{|\nabla\d|^2}{2}\right)\cdot\u\ d\x=0.$$
Multiplying \eqref{e22} by $-(\triangle\d+|\nabla\d|^2\d)$ and
integrating over $\Omega$, we obtain
\begin{equation*}
-\int_\O\partial_t\d \cdot\triangle\d\ d\x-\int_\O
(\u\cdot\nabla\d)\cdot \triangle\d \ d\x=-\int_\O
|\triangle\d+|\nabla\d|^2\d|^2 d\x.
\end{equation*}
Here we have used the fact that $|\d|=1$ to get
$$\left(\partial_t\d+\u\cdot\nabla\d\right)\cdot|\nabla\d|^2\d
=\frac{1}{2}\left(|\nabla\d|^2\partial_t(|\d|^2)+\u\cdot\nabla|\d|^2|\nabla\d|^2\right)=0.$$
Since $\partial_{\bf\nu} \d=0$ on $\partial\O$, integrating by
parts, we have $$\int_\O\partial_t \d\cdot\triangle\d \
d\x=-\frac{1}{2}\frac{d}{dt}\int_\O|\nabla\d|^2\ d\x.$$ Hence we
obtain
\begin{equation}\label{1711}
\frac{1}{2}\f{d}{dt}\int_\O|\nabla\d|^2\ d\x+\int_\O
|\triangle\d+|\nabla\d|^2\d|^2 \ d\x=\int_\O (\u\cdot\nabla\d)\cdot
\triangle\d \ d\x.
\end{equation}
By adding \eqref{514} and \eqref{1711}, we eventually get the
identity:
\begin{equation}\label{167}
\frac{1}{2}\f{d}{dt}\int_\O(\r|\u|^2+|\nabla\d|^2)\ d\x+\int_\O
(|\nabla\u|^2+|\triangle\d+|\nabla\d|^2\d|^2)\ d\x=0.
\end{equation}
Since $\nabla\d\in L^2(0,T_0;H^1(\O))$ and $|\d|=1$, we have
\begin{equation}\label{1670}
\triangle\d\cdot\d+|\nabla\d|^2=0,
\end{equation}
 and then
\begin{equation*}
\begin{split}
\int_\O|\triangle\d+|\nabla\d|^2\d|^2)\ d\x
%&=\int_\O
%(|\triangle\d|^2-|\nabla\d|^4)\ d\x=\int_\O\left(|\triangle\d|^2-(\triangle\d\cdot\d)^2\right)\ d\x\\&
=\int_\O
|\triangle\d\times \d|^2\ d\x.
\end{split}
\end{equation*}
Now, by virtue of the Poincar\'{e} inequality
$\|\nabla\u\|_{L^2}^2\geq \lambda_1\|\u\|_{L^2}^2$, we get
\begin{equation*}
\frac{1}{2}\f{d}{dt}\left(\|\sqrt{\r}\u\|_{L^2}^2+\|\nabla\d\|^2_{L^2}\right)+\frac{\lambda_1}{\hat{\r}}\|\sqrt{\r}\u\|_{L^2}^2
\leq 0,
\end{equation*}
i.e.,
\begin{equation}\label{in}
\f{d}{dt}\left(e^{\frac{2\lambda_1}{\hat{\r}}t}\|\sqrt{\r}\u\|_{L^2}^2\right)+e^{\frac{2\lambda_1}{\hat{\r}}t}\f{d}{dt}\|\nabla\d\|^2_{L^2}
\leq 0.
\end{equation}
 Integrating \eqref{in} from $0$ to $t$, we obtain
\begin{equation*}
\begin{split}
&e^{\frac{2\lambda_1}{\hat{\r}}
t}\left(\|(\sqrt{\r}\u)(t)\|_{L^2}^2+\|\nabla\d(t)\|_{L^2}^2\right)\\
&\leq
\|\sqrt{\r_0}\u_0\|_{L^2}^2+\|\nabla\d_0\|_{L^2}^2+\frac{2\lambda_1}{\hat{\r}}\int_0^t
e^{\frac{2\lambda_1}{\hat{\r}} \tau}
\|\nabla\d(\tau)\|_{L^2}^2d\tau.
\end{split}
\end{equation*}
It follows from  Gr\"{o}nwall's inequality that
\begin{equation*}
e^{\frac{2\lambda_1}{\hat{\rho}}t}\|\nabla\d\|_{L^2}^2\leq
(\|\sqrt{\r_0}\u_0\|_{L^2}^2+\|\nabla\d_0\|_{L^2}^2)(1+\frac{2\lambda_1}{\hat{\rho}}te^{\frac{2\lambda_1}{\hat{\rho}}t}),
\end{equation*}
and furthermore,
\begin{equation*}
e^{\frac{2\lambda_1}{\hat{\rho}}t}(\|\sqrt{\r}\u\|_{L^2}^2+\|\nabla\d\|_{L^2}^2)\leq
(\|\sqrt{\r_0}\u_0\|_{L^2}^2+\|\nabla\d_0\|_{L^2}^2)(2+\frac{2\lambda_1}{\hat{\rho}}te^{\frac{2\lambda_1}{\hat{\rho}}t}
-e^{\frac{2\lambda_1}{\hat{\rho}}t}+\frac{2\lambda_1}{\hat{\rho}}t).
\end{equation*}
\end{proof}

Usually, \eqref{167} is called the basic energy law governing the
system \eqref{e2}-\eqref{bc}. It reflects the energy dissipation
property of the flow of liquid crystals.

\subsection{A more explicit lower bound for the existence time}

 We denote by $T^*$ the maximal existence time for
$(\r, \u, P, \d)$ which means $(\r, \u, P, \d)$ cannot be continued
beyond $T^*$ into a strong solution of \eqref{e2}-\eqref{bc}. Let us
first state a continuation criterion:

\begin{Lemma}\label{516}
Let $\r_0, \u_0 , \d_0$ be as in Theorem \ref{T1} and assume that
system \eqref{e2} with the initial-boundary conditions
\eqref{ic}-\eqref{bc} has a strong solution on a finite time interval
$[0,T^\ast)$ with
$$\r\in L^\infty(0,T^\ast; W^{1,r}(\O)),\ \inf_{t<T^\ast,\ \x\in \Omega}\r(\x,t)>0,$$
$$\u\in L^\infty(0,T^\ast; D^{1-\frac{1}{p},p}_{A_q}) \quad \textrm{and}
\quad \d\in L^\infty(0,T^\ast; B^{3(1-\frac{1}{p})}_{q, p}).$$ Then
$(\r, \u, P, \d)$ can be continued beyond $T^\ast$ into a strong
solution of \eqref{e2}-\eqref{bc}.
\end{Lemma}

\begin{proof}
 Indeed, a positive lower bound $\check{T}$ for the existence
time has\ already been obtained in the proof of Theorem \ref{T1}
\big(see \eqref{TTT1}\eqref{TTT2} and \eqref{10}\big) when $(\r_0,
\u_0, \d_0)$ remains in a bounded set of $$W^{1,r}\times
D^{1-\frac{1}{p},p}_{A_q}\times B^{3(1-\frac{1}{p})}_{q, p}$$ with
in addition $\inf_{\x\in \Omega}\r_0(\x)\geq\check{\r}$ for a fixed
$\check{\r}>0$. Hence system \eqref{e2} with initial density
$\r(T^\ast-\frac{\check{T}}{2})$, initial velocity
$\u(T^\ast-\frac{\check{T}}{2})$ and initial orientation field
$\d(T^\ast-\frac{\check{T}}{2})$ has a unique strong solution on
$[0, \check{T}]$ which provides a continuation of the strong
solution beyond $T^\ast$.
\end{proof}

Combining Lemma \ref{51} and Lemma \ref{516}   enables us to get
the following result:

\begin{Proposition}\label{517}
Let $\r_0, \u_0 , \d_0$ be as in Theorem \ref{T1} and let $(\r,
\u,P, \d )$ denote the corresponding strong solution of
\eqref{e2}-\eqref{bc}. Then there exists some constant $C$ depending
on $p, q, r, \mu, \lambda,\gamma, \Omega\ \textrm{and}\ \check{\r}$,
such that, the maximal existence time $T^\ast$ for $(\r, \u, P, \d)$
satisfies
$$T^\ast\geq\frac{C}{(1+\|\r_0\|_{W^{1,r}})^\kappa (U^0)^\iota}$$
for some positive exponents $\kappa$ and $\iota$ depending only on
the regularity parameters.
\end{Proposition}

\begin{proof}
Fix a $\tilde{T}<T^\ast$. We aim at proving that if $\tilde{T}\leq
C(1+\|\r_0\|_{W^{1,r}})^{-\kappa} (U^0)^{-\iota}$ for a convenient
choice of $C, \kappa\ \textrm{and}\  \iota$ then $(\r, \u, P, \d)$
may be bounded in $M_{\tilde{T}}^{p, q, r}$ by a function depending
only on the data. Then Lemma \ref{516} will entail Proposition
\ref{517}.

Define
\begin{equation*}
\begin{split}
G(t):=&\|\u \|_{L^\infty_t (D_{A_q}^{1-\f{1}{p},p})}+\|\u\|_{L^p_t(
W^{2,q})}+\|\partial_t\u\|_{L^p_t (L^q)}\\
&+\|\nabla P\|_{L^p_t(L^q)}+\|\d\|_{L^\infty_t(
B_{q,p}^{3(1-\f{1}{p})})}+\|\d\|_{\mathcal{W}_{q,p}(0,t)}.
\end{split}
\end{equation*}
According to Theorems \ref{T3}-\ref{T4} and \eqref{5000}, we have
\begin{equation}\label{518}
\begin{split}
G(t)&\leq
C\Big(\mathcal{B}_{\r}^2(t)\big(\|\u_0\|_{D_{A_q}^{1-\f{1}{p},p}}
+\|\u\cdot\nabla\u\|_{L_t^p(L^q)}+\|\nabla\cdot\big((\nabla\d)^\top\nabla\d)\|_{L_t^p(L^q)}\big)\\
&\qquad\quad+\mathcal{C}_{\r}(t)\|\u\|_{L_t^p(L^q)}
+\|\d_0\|_{B_{q,p}^{3(1-\f{1}{p})}}+\|-\u\cdot\nabla\d+|\nabla\d|^2\d\|_{L_t^p(L^q)}\Big),
\end{split}
\end{equation}
where $C=C(p, q, r, \Omega, \check{\r}, \hat{\r})$.

Combining the Gagliardo-Nirenberg-Sobolev inequality and Young's
inequality yields, for all $\varepsilon>0,$
\begin{equation}\label{519}
\|\u\|_{L^q}\leq
C\big(\varepsilon\|\u\|_{W^{2,q}}+\varepsilon^{1-\frac{1}{\theta}}\|\u\|_{L^2}\big)
\quad \textrm{with}\ \ \theta=\frac{4q}{7q-6}.
\end{equation}
We note that Lemma \ref{51} insures
$$\|\u\|_{L_t^\infty(L^2)}\leq C(1+t^\frac{1}{2})U^0,$$
then employing H\"{o}lder's inequality, we have
\begin{equation}\label{521}
\begin{split}
\|\u\|_{L_t^p(L^2)}&\leq t^\frac{1}{p}\|\u\|_{L_t^\infty(L^2)}\leq
Ct^\frac{1}{p}(1+t^\frac{1}{2})U^0.
\end{split}
\end{equation}
Moreover, We get
\begin{equation}\label{5201}
\begin{split}
\|\u\cdot\nabla\u\|_{L_t^p(L^q)}\leq
\|\u\|_{L_t^\infty(L^q)}\|\nabla\u\|_{L_t^p(L^\infty)} \leq
Ct^{\frac{1}{2}-\frac{3}{2q}}G^2(t),
\end{split}
\end{equation}
\begin{equation}\label{5202}
\begin{split}
\|\nabla\cdot\big((\nabla\d)^\top\nabla\d)\|_{L^p_t(L^q)}\le
C\|\nabla\d\|_{L_t^\infty(L^q)}\|\triangle\d\|_{L_t^p(L^\infty)} \le
 Ct^{\f{1}{3}-\f{1}{q}}G^2(t),
\end{split}
\end{equation}
\begin{equation}\label{5203}
\begin{split}
\|-\u\cdot\nabla\d+|\nabla\d|^2\d\|_{L^p_t(L^q)}&\le\|\u\|_{L_t^\infty(L^q)}\|\nabla\d\|_{L_t^p(L^\infty)}
+\|\nabla\d\|_{L_t^\infty(L^q)}\|\nabla\d\|_{L_t^p(L^\infty)}\\
&\le Ct^{\f{1}{2}-\f{3}{2q}}G^2(t).
\end{split}
\end{equation}
Here we have used the fact that $|\d|=1$. Hence plugging
\eqref{519}-\eqref{5203} in \eqref{518} while taking
$\varepsilon=\epsilon \mathcal{C}^{-1}_\r(t)$ with $\epsilon$
suitably small, we get
\begin{equation}\label{520}
\begin{split}
G(t)&\leq
C\Big(\mathcal{B}_{\r}^2(t)\big(U^0+(t^{\frac{1}{2}-\frac{3}{2q}}+t^{\frac{1}{3}-\frac{1}{q}})G^2(t)\big)
+\mathcal{C}_{\r}^{\frac{1}{\theta}}(t)t^{\frac{1}{p}}(1+t^\frac{1}{2})U^0+U^0+t^{\frac{1}{2}-\frac{3}{2q}}G^2(t)\Big).
\end{split}
\end{equation}
 On the other hand, using the same argument as for $\r^{k+1}$ in Subsection 4.2, we obtain
\begin{equation}\label{5211}
\|\nabla\r\|_{L_t^\infty(L^r)}\leq
\|\r_0\|_{W^{1,r}}e^{Ct^{\frac{3}{2}-\frac{1}{p}-\frac{3}{2q}}G(t)},
\end{equation}
\begin{equation}\label{5212}
\begin{split}
\|\r\|_{C_t^\beta(L^\infty)}&\leq
C\big(\|\r\|_{L_t^\infty(W^{1,r})}+\|\partial_t\r\|_{L_t^\infty(L^s)}\big)\\
&\leq
C\big(\|\r\|_{L_t^\infty(W^{1,r})}+\|\u\|_{L^\infty_t(L^q)}\|\nabla\r\|_{L_t^\infty(L^r)}\big)\\
&\leq C\|\nabla\r\|_{L_t^\infty(L^r)}\big(1+G(t)\big).
\end{split}
\end{equation}
Then, according to the definitions of $\mathcal{B}_\r(t)$ and
$\mathcal{C}_\r(t)$ in Theorem 3.7 in \cite{D}, using \eqref{5211}
and \eqref{5212}, we eventually get,
\begin{equation}\label{522}
\mathcal{B}_\r(t)\leq
Ce^{Ct^{\frac{3}{2}-\frac{1}{p}-\frac{3}{2q}}G(t)}(1+\|\r_0\|_{W^{1,r}})^{\frac{r}{r-3}},
\end{equation}
\begin{equation}\label{5221}
\mathcal{C}_\r(t)\leq
Ce^{Ct^{\frac{3}{2}-\frac{1}{p}-\frac{3}{2q}}G(t)}\Big(\big(1+\|\r_0\|_{W^{1,r}}\big)^{\gamma_1}
+\big(1+\|\r_0\|_{W^{1,r}}\big)^{\gamma_2}\|\r_0\|_{W^{1,r}}^{\frac{1}{\beta}}\big(1+G(t)\big)^{\frac{1}{\beta}}\Big),
\end{equation}
 where $\gamma_1$ and $\gamma_2$ depend only on $p, q, r\ \textrm{and} \
\beta$.

Plugging \eqref{522}-\eqref{5221} in \eqref{520}, for some positive
exponents $\delta_1$ and $\delta_2$, we have
\begin{equation*}
\begin{split}
G(t)\leq &
Ce^{Ct^{\frac{3}{2}-\frac{1}{p}-\frac{3}{2q}}G(t)}(1+\|\r_0\|_{W^{1,r}})^{\delta_1}\Big(U^0\big(1+t^{\frac{1}{p}}(1+t^{\frac{1}{2}})(1+G(t))^{\delta_2}\big)
\\&+(t^{\frac{1}{2}-\frac{3}{2q}}+t^{\frac{1}{3}-\frac{1}{q}})G^2(t)\Big).
\end{split}
\end{equation*}

Assume that $\tilde{T}$ has been chosen such that
\begin{equation}\label{524}
 G(\tilde{T})\leq 8C(1+\|\r_0\|_{W^{1,r}})^{\delta_1}U^0.
\end{equation}
 This
is possible because of the continuity of the function $t\mapsto
G(t)$. Noticing that $G(t)$ is increasing in $t$, then a standard
induction argument shows \eqref{524} is satisfied at time
$t\leq\tilde{T}$ with a strict inequality whenever the following
three inequalities are satisfied:
$$8C^2(1+\|\r_0\|_{W^{1,r}})^{\delta_1}U^0t^{\frac{3}{2}-\frac{1}{p}-\frac{3}{2q}}<\ln 2,$$
$$\big(1+8C(1+\|\r_0\|_{W^{1,r}})^{\delta_1}U^0\big)^{\delta_2}t^{\frac{1}{p}}(1+t^{\frac{1}{2}})\leq 1,$$
$$64C^2(1+\|\r_0\|_{W^{1,r}})^{2\delta_1}U^0(t^{\frac{1}{2}-\frac{3}{2q}}+t^{\frac{1}{3}-\frac{1}{q}})\leq 2.$$
Hence Lemma \ref{516} enables us to continue the solution $(\r, \u,
P, \d)$ beyond $\tilde{T}$.

The proof of Proposition \ref{517} is complete.
\end{proof}

\subsection{The case of a small initial velocity and orientation field}

Proposition \ref{517} insures that the existence time of a strong
solution for \eqref{e2}-\eqref{bc} goes infinity (for fixed initial
density) when $\u_0$ (resp. $\d_0$) tends to $0$ in
$D_{A_q}^{1-\f{1}{p},p}$ (resp. $B^{3(1-\frac{1}{p})}_{q, p}$). We
now aim at stating that the system has indeed a global strong
solution if $\u_0$
 and $\d_0$ are suitably small. This will give Theorem \ref{T2}.

 Let $(\r, \u, P, \d)$ be the strong solution given by Theorem \ref{T1}. For any $\zeta\geq0$, define
$$G_{0, 2, \zeta}(t):=(\|\sqrt{\r_0}\u_0\|_{L^2}+\|\nabla\d_0\|_{L^2})\left(1+(2\zeta t)^{\frac{1}{2}}e^{\zeta t}\right)\quad \textrm{and}\ \ G_{0, 2}:=G_{0, 2,0}(t).$$
By Lemma \ref{51}, for $t<T^\ast$, we have
\begin{equation}\label{525}
\|\sqrt{\r}\u\|_{L_t^p(L^2)}\leq CG_{0,
2}t^{\frac{1}{p}}\big(1+(\frac{2\lambda_1}{\hat{\r}}t)^{\frac{1}{2}}\big)
\end{equation}
and
$$\|(\sqrt{\r}\u)(t)\|_{L^2}+\|\nabla\d(t)\|_{L^2}\leq C e^{-\frac{\lambda_1}{\hat{\r}}
t}G_{0, 2,\frac{\lambda_1}{\hat{\rho}}}(t)$$.

 Hence, starting from \eqref{518}, using
\eqref{519},  \eqref{525} and the fact that
\begin{equation*}
\begin{split}
\|\u\cdot\nabla\u\|_{L_t^p(L^q)}&\leq \|\u\|_{L_t^\infty(L^q)}\|\nabla\u\|_{L_t^p(L^\infty)}\\
&\leq
C\|\u\|_{L_t^\infty(D_{A_q}^{1-\frac{1}{p},p})}\|\u\|_{L_t^p(W^{2,q})}
\leq CG^2(t),
\end{split}
\end{equation*}
\begin{equation*}
\begin{split}
\|\nabla\cdot\big((\nabla\d)^\top\nabla\d)\|_{L^p_t(L^q)}&\le
C\|\nabla\d\|_{L_t^\infty(L^q)}\|\triangle\d\|_{L_t^p(L^\infty)}\\
&\leq
C\|\d\|_{L_t^\infty(B_{q,p}^{3(1-\frac{1}{p})})}\|\d\|_{L_t^p(W^{3,q})}
\leq CG^2(t),
\end{split}
\end{equation*}
\begin{equation*}
\begin{split}
&\|-\u\cdot\nabla\d+|\nabla\d|^2\d\|_{L_t^p(L^q)}\\&\leq
\|\u\|_{L_t^\infty(L^q)}\|\nabla\d\|_{L_t^p(L^\infty)}
+\|\nabla\d\|_{L_t^\infty(L^q)}\|\nabla\d\|_{L_t^p(L^\infty)}\\
&\leq
C\Big(\|\u\|_{L_t^\infty(D_{A_q}^{1-\frac{1}{p},p})}\|\d\|_{L_t^p(W^{2,q})}+\|\d\|_{L_t^\infty(B_{q,p}^{3(1-\frac{1}{p})})}\|\d\|_{L_t^p(W^{2,q})}\Big)\leq
CG^2(t),
\end{split}
\end{equation*}
we end up with
\begin{equation*}
G(t)\leq
C\Big(\big(1+\mathcal{B}_\r^2(t)\big)\big(U^0+G^2(t)\big)+\mathcal{C}_\r^{\frac{1}{\theta}}(t)G_{0,2}t^{\frac{1}{p}}\big(1+
(\frac{2\lambda_1}{\hat{\r}} t)^{\frac{1}{2}}\big)\Big).
\end{equation*}

Once again, the bounds for $\mathcal{B}_\r(t)$ and
$\mathcal{C}_\r(t)$ will follow from \eqref{522} and \eqref{5221}.
However, in contrast with the previous section, we are going to take
advantage of Lemma \ref{51} to avoid the appearance of the factor
$t^{\frac{3}{2}-\frac{1}{p}-\frac{3}{2q}}$. Indeed, since
$\big(L^2(\Omega), W^{2,q}(\Omega)\big)_\vartheta\hookrightarrow
W^{1,\infty}(\O)$ with $\vartheta=\frac{5q}{7q-6},$ then it follows
from H\"{o}lder's inequality that
\begin{equation*}
\begin{split}
\int_0^t\|\nabla\u(\tau)\|_{L^\infty} d\tau&\leq C\int_0^t\|\u(\tau)\|_{L^2}^{1-\vartheta}\|\u(\tau)\|_{W^{2,q}}^\vartheta d\tau\\
&\leq C\int_0^t\big(e^{-\frac{\lambda_1}{\hat{\r}}\tau}G_{0, 2, \frac{\lambda_1}{\hat{\r}}}(\tau)\big)^{1-\vartheta}\|\u(\tau)\|_{W^{2,q}}^\vartheta d\tau\\
&\leq
CG_{0,2,\frac{\lambda_1}{\hat{\r}}}^{1-\vartheta}(t)G^\vartheta(t).
\end{split}
\end{equation*}
Now, bounding $\mathcal{B}_\r(t)$ and $\mathcal{C}_\r(t)$ may be
done by mimicking the proof of Proposition \ref{517} and we
eventually conclude that
\begin{equation}\label{528}
\begin{split}
G(t)\leq
Ce^{CG_{0,2,\frac{\lambda_1}{\hat{\r}}}^{1-\vartheta}(t)G^\vartheta(t)}(1+\|\r_0\|_{W^{1,r}})^{\delta_3}\Big(U^0\big(1+t^{\frac{1}{p}}(1+t^{\frac{1}{2}})(1+G(t))^{\delta_4}\big)+G^2(t)\Big)
\end{split}
\end{equation}
for some positive exponents $\delta_3$ and $\delta_4$ depending only
on $p, q, r$.

Fix a positive $\bar{T}$ and assume that
\begin{equation}\label{529}
G(t)\leq 8C(1+\|\r_0\|_{W^{1,r}})^{\delta_3}U^0, \quad t\in [0,
\bar{T}].
\end{equation}
If the data are so small that
$$CG^{1-\vartheta}_{0,2,\frac{\lambda_1}{\hat{\r}}}(\bar{T})\big(8C(1+\|\r_0\|_{W^{1,r}})^{\delta_3}U^0\big)^\vartheta\leq \ln 2,$$
then \eqref{528} implies
$$G(t)\leq 2C(1+\|\r_0\|_{W^{1,r}})^{\delta_3}\Big(U^0\big(1+t^{\frac{1}{p}}(1+t^{\frac{1}{2}})(1+G(t))^{\delta_4}\big)+G^2(t)\Big).$$
Now, if in addition
$$64C^2(1+\|\r_0\|_{W^{1,r}})^{2\delta_3}U^0\leq \frac{1}{2}\ \ \textrm{and}\ \ \bar{T}^{\frac{1}{p}}(1+\bar{T}^{\frac{1}{2}})\big(1+8C(1+\|\r_0\|_{W^{1,r}})^{\delta_3}U^0\big)^{\delta_4}\leq \frac{3}{2},$$
then \eqref{529} is satisfied with the constant $6C$ instead of
$8C$. A standard bootstrap argument enables to conclude to the
second part of Theorem \ref{T1}.

\bigskip

%%%%%%%%%%%%%%%%%%%%
\section{Weak-Strong Uniqueness}

The purpose of this section is to show $\textit{Weak-Strong
Uniqueness}$ in Theorem \ref{T2}. To this end, we need to obtain
first an energy estimate for the strong solution to system
\eqref{e2}-\eqref{bc}.
\begin{Lemma}\label{l4}
Let $p,q,r$ satisfy the same conditions as Theorem \ref{T1} and
$(\r,\u, P, \d)\in M_{T_0}^{p,q,r}$ be the unique solution to
\eqref{e2}-\eqref{bc} on $\O\times [0, T_0]$. Then for any $0<t\leq
T_0$, we have
\begin{equation}\label{l6}
\begin{split}
&\frac{1}{2}\int_\O\left(\r(t)|\u(t)|^2+|\nabla\d(t)|^2\right)
d\x+\int_0^t\int_\O (|\nabla\u|^2+|\triangle\d+|\nabla\d|^2\d|^2)\
d\x d\tau\\& =\frac{1}{2}\int_\O(\r_0|\u_0|^2+|\nabla\d_0|^2)\ d\x.
\end{split}
\end{equation}
\end{Lemma}

\begin{proof}
Integrating \eqref{167} over the time interval $[0,t]$, we obtain
the energy equality \eqref{l6}.
\end{proof}

\vspace{3mm}

%The following proposition gives a kind of uniqueness result. It is similar to the result which Serrin used in the Navier-Stokes equations (\cite{JS1,JS2}), also which Lin-Liu used in the Leslie system of variable length (\cite{LL}).
%\begin{Proposition}\label{ws}
%Let $(\tilde{\u},\tilde{\d})$ be another solution that satisfies the following energy inequality:
%\begin{equation}\label{260}
%\begin{split}
%&\int_\O(|\tilde{\u}(t)|^2+|\nabla\tilde{\d}(t)|^2) d\x+2\int_0^t\int_\O
%(|\nabla\tilde{\u}|^2+|\triangle\tilde{\d}+|\nabla\tilde{\d}|^2\tilde{\d}|^2) d\x ds\\&
%\leq\int_\O(|\u_0|^2+|\nabla\d_0|^2)d\x.
%\end{split}
%\end{equation}
%Then $(\u, \d)\equiv(\tilde{\u},\tilde{\d})$.
%\end{Proposition}

%Now, suppose  $(\tilde{\u},\tilde{\d},\Pi)$
%such that $$\tilde{\u}\in L^{2,\infty}(\O\times [0,T_0])\cap W_2^{1,0}(\O_{T_0}),$$
%$$\tilde{\d}\in L^\infty ([0,T_0],H^1(\O))\cap L^2([0,T_0],H^2(\O)),$$ $$\nabla\Pi\in L^2([0, T_0],L^1(\O))$$ is a weak solution to \eqref{e2}-\eqref{bc}, we have (cf. \cite{LLW} Section 4), for (almost) all $t\in(0,T_0)$,
%\begin{equation}\label{26}
%\begin{split}
%&\int_\O(|\tilde{\u}(t)|^2+|\nabla\tilde{\d}(t)|^2) d\x+2\int_0^t\int_\O
%(|\nabla\tilde{\u}|^2+|\triangle\tilde{\d}+|\nabla\tilde{\d}|^2\tilde{\d}|^2) d\x ds\\&
%\leq\int_\O(|\u_0|^2+|\nabla\d_0|^2)d\x.
%\end{split}
%\end{equation}
Now, we proceed to prove weak-strong uniqueness. Let
$(\tilde{\r},\tilde{\u},\Pi,\tilde{\d})$ be a global (in time) weak
solution. On   one hand, as the density $\tilde{\r}$ satisfies
\begin{equation*}
\begin{cases}
\partial_t\tilde{\r}+\tilde{\u}\cdot\nabla\tilde{\r}=0\\
\tilde{\r}|_{t=0}=\r_0\in W^{1,r}(\O)
\end{cases}
\end{equation*}
with $\tilde{\u}\in\big(L^2_{\textrm{loc}}(\R^+;H_0^1(\O))\big)^3$,
Theorem 1 in \cite{DB} insures that $\tilde{\r}\in C(\R^+; W^{1,r^-}(\O))$ for all $r^-< r$.
On the other hand, we remark that, in view of the regularity of the
strong solution $(\r, \u, P, \d)$, we deduce from the weak
formulation that
\begin{equation}\label{w11}
\begin{split}
&\int_\O\tilde{\r}\tilde{\u}\cdot\u\ d\x+\int_0^t\int_\O\nabla\tilde{\u}:\nabla\u\ d\x d\tau\\
&=\int_\O\r_0|\u_0|^2\
d\x+\int_0^t\int_\O\tilde{\r}\tilde{\u}\cdot(\partial_\tau\u+\tilde{\u}\cdot\nabla\u)\
d\x
d\tau-\int_0^t\int_\O(\nabla\tilde{\d})^\top\triangle\tilde{\d}\cdot\u\
d\x d\tau
\end{split}
\end{equation}
and
\begin{equation}\label{w21}
\begin{split}
&\int_\O\nabla\tilde{\d}:\nabla\d\ d\x-\int_\O|\nabla\d_0|^2\ d\x\\
&=\int_0^t\!\!\int_\O\big(-\tilde{\d}\cdot
\triangle\d_\tau+\tilde{\u}\cdot\nabla\tilde{\d}\cdot
\triangle\d-\triangle\tilde{\d}\cdot\triangle\d-|\nabla\tilde{\d}|^2\tilde{\d}\cdot
\triangle\d\big)\ d\x d\tau
\end{split}
\end{equation}
for  almost all $t\in(0,T_0)$.

If we write
\begin{equation}\label{w13}
\begin{split}
&\tilde{\r}\partial_t\u+\tilde{\r}\tilde{\u}\cdot\nabla\u-\triangle\u+\nabla P\\
&=(\tilde{\r}-\r)(\partial_t\u+\u\cdot\nabla\u)+\tilde{\r}(\tilde{\u}-\u)\cdot\nabla\u-\nabla\left(\frac{|\nabla\d|^2}{2}\right)-(\nabla\d)^\top\triangle\d,
\end{split}
\end{equation}
then multiply \eqref{w13} by $\tilde{\u}$ and integrate over $\O\times (0,t)$ to find
\begin{equation}\label{w14}
\begin{split}
&\int_0^t\int_\O(\tilde{\r}\partial_\tau\u+\tilde{\r}\tilde{\u}\cdot\nabla\u)\cdot\tilde{\u}\ d\x d\tau+\int_0^t\int_\O\nabla\u:\nabla\tilde{\u}\ d\x d\tau\\
&=\int_0^t\int_\O\big((\tilde{\r}-\r)(\partial_\tau\u+\u\cdot\nabla\u)\cdot\tilde{\u}+\tilde{\r}(\tilde{\u}-\u)\cdot\nabla\u\cdot\tilde{\u}-(\nabla\d)^\top\triangle\d\cdot\tilde{\u}\big)
\ d\x d\tau,
\end{split}
\end{equation}
and meanwhile, replace $\partial_\tau \d$ by \eqref{e23} in
\eqref{w21} to get
\begin{equation}\label{w22}
\begin{split}
&\int_\O\nabla\tilde{\d}:\nabla\d\ d\x-\int_\O|\nabla\d_0|^2\ d\x\\
&=\int_0^t\!\!\int_\O\big(-2\triangle\tilde{\d}\cdot
\triangle\d+\u\cdot\nabla\d\cdot\triangle\tilde{\d}+\tilde{\u}\cdot\nabla\tilde{\d}\cdot\triangle\d-|\nabla\d|^2\d\cdot
\triangle\tilde{\d}-|\nabla\tilde{\d}|^2\tilde{\d}\cdot\triangle\d\big)\
d\x d\tau.
\end{split}
\end{equation}
Combining \eqref{w11}, \eqref{w14} and \eqref{w22}, we get for
almost all $t\in (0,T_0)$,
\begin{equation}\label{w15}
\begin{split}
&\int_\O(\tilde{\r}\tilde{\u}\cdot\u+\nabla\tilde{\d}:\nabla\d)\ d\x+2\int_0^t\int_\O(\nabla\tilde{\u}:\nabla\u+\triangle\tilde{\d}\cdot\triangle\d)\ d\x d\tau\\
&=\int_\O(\r_0|\u_0|^2+|\nabla\d_0|^2)\
d\x-\int_0^t\int_\O(\nabla\tilde{\d})^\top\triangle\tilde{\d}\cdot\u\
d\x
d\tau\\
&\quad+\int_0^t\int_\O\big((\tilde{\r}-\r)(\partial_\tau\u+\u\cdot\nabla\u)\cdot\tilde{\u}+\tilde{\r}(\tilde{\u}-\u)\cdot\nabla\u\cdot\tilde{\u}-(\nabla\d)^\top\triangle\d\cdot\tilde{\u}\big) \ d\x d\tau\\
&\quad
+\int_0^t\int_\O(\u\cdot\nabla\d\cdot\triangle\tilde{\d}+\tilde{\u}\cdot\nabla\tilde{\d}\cdot\triangle\d-|\nabla\d|^2\d\cdot
\triangle\tilde{\d}-|\nabla\tilde{\d}|^2\tilde{\d}\cdot\triangle\d)\
d\x d\tau.
\end{split}
\end{equation}
%We multiply \eqref{w13} by $\u$ and integrate over $\O\times (0,t)$ to find
%\begin{equation}\label{w16}
%\begin{split}
%&\frac{1}{2}\int_\O\tilde{\r}|\u|^2\ d\x+\int_0^t\int_\O|\nabla\u|^2\ d\x d\tau\\
%&=\frac{1}{2}\int_\O\r_0|\u_0|^2\ d\x\\
%&\quad+\int_0^t\int_\O\big((\tilde{\r}-\r)(\partial_\tau\u+\u\cdot\nabla\u)\cdot\u+\tilde{\r}(\tilde{\u}-\u)\cdot\nabla\u\cdot\u-(\nabla\d)^\top\triangle\d\cdot\u\big)
%\ d\x d\tau.
%\end{split}
%\end{equation}
%Here we have used the fact that
%$\partial_t\tilde{\r}=-\tilde{\u}\cdot\nabla\tilde{\r}$. And, we
%multiply \eqref{w23} by $-\triangle\d$ and integrate over $\O\times
%(0,t)$ to find
%\begin{equation}\label{w24}
%\begin{split}
%&\frac{1}{2}\int_\O|\nabla\d|^2\
%d\x-\frac{1}{2}\int_\O|\nabla\d_0|^2\
%d\x-\int_0^t\int_\O(\u\cdot\nabla\d)\cdot\triangle\d\ d\x d\tau\\
%&=-\int_0^t\int_\O|\triangle\d|^2 \ d\x
%d\tau-\int_0^t\int_\O|\nabla\d|^2\d\cdot\triangle\d \ d\x d\tau.
%\end{split}
%\end{equation}
%Combining \eqref{w16} and \eqref{w24},
From \eqref{w13} and \eqref{e23}, using the same argument as to get
the energy estimate \eqref{l6},
 we get for almost all $t\in
(0,T_0)$,
\begin{equation}\label{w17}
\begin{split}
&\frac{1}{2}\int_\O(\tilde{\r}|\u|^2+|\nabla\d|^2)\ d\x+\int_0^t\int_\O(|\nabla\u|^2+|\triangle\d+|\nabla\d|^2\d|^2)\ d\x d\tau\\
&=\frac{1}{2}\int_\O(\r_0|\u_0|^2+|\nabla\d_0|^2)
d\x+\int_0^t\int_\O\big((\tilde{\r}-\r)(\partial_\tau\u+\u\cdot\nabla\u)\cdot\u+\tilde{\r}(\tilde{\u}-\u)\cdot\nabla\u\cdot\u\big)
d\x d\tau.
\end{split}
\end{equation}
Here we have used the fact that
$\partial_t\tilde{\r}=-\tilde{\u}\cdot\nabla\tilde{\r}$.

 Then,  adding \eqref{26} and \eqref{w17} and substracting
\eqref{w15}, together with the fact that \eqref{1670} and
$(\tilde{\r},\tilde{\u},\Pi,\tilde{\d})$ is a weak solution,
%\begin{equation}\label{w2}
%\begin{split}
%&\frac{1}{2}\int_{\O}|\nabla\tilde{\d}|^2\
%d\x-\frac{1}{2}\int_{\O}|\nabla\d_0|^2\
%d\x-\int_0^t\!\!\int_\O(\tilde{\u}\cdot\nabla\tilde{\d})\cdot\triangle\tilde{\d}\
%d\x ds\\&=-\int_0^t\!\!\int_\O|\triangle\tilde{\d}|^2\ d\x
%ds-\int_0^t\!\!\int_\O
%|\nabla\tilde{\d}|^2\tilde{\d}\cdot\triangle\tilde{\d}\ d\x ds.
%\end{split}
%\end{equation}
%\begin{equation}\label{w3}
%\begin{split}
%&\frac{1}{2}\int_\O|\nabla\tilde{\d}|^2
%d\x-\frac{1}{2}\int_\O|\nabla\d_0|^2 d\x-\int_0^t\!\!\int_\O
%(\tilde{\u}\cdot\nabla\tilde{\d})\cdot \triangle\tilde{\d} \ d\x ds\\
%&=-\int_0^t\!\!\int_\O
%|\triangle\tilde{\d}+|\nabla\tilde{\d}|^2\tilde{\d}|^2 \ d\x ds.
%\end{split}
%\end{equation}
 we obtain
\begin{equation*}
\begin{split}
&\frac{1}{2}\int_\O(\tilde{\r}|\u-\tilde{\u}|^2+|\nabla\d-\nabla\tilde{\d}|^2)\ d\x+\int_0^t\int_\O(|\nabla\u-\nabla\tilde{\u}|^2+|\triangle\d-\triangle\tilde{\d}|^2)\ d\x d\tau\\
&\leq\int_0^t\int_\O\big((\tilde{\r}-\r)(\partial_\tau\u+\u\cdot\nabla\u)\cdot(\u-\tilde{\u})-\tilde{\r}(\u-\tilde{\u})\cdot\nabla\u\cdot(\u-\tilde{\u})\big) \ d\x d\tau\\
&\quad-\int_0^t\!\!\int_\O\big((\nabla\d-\nabla\tilde{\d})\cdot\triangle\d\cdot(\u-\tilde{\u})-\u\cdot(\nabla\d-\nabla\tilde{\d})\cdot(\triangle\d-\triangle\tilde{\d})\\
&\qquad\qquad\quad
+(|\nabla\d|^2\d-|\nabla\tilde{\d}|^2\tilde{\d})\cdot(\triangle\d-\triangle\tilde{\d})\big)\
d\x d\tau.
\end{split}
\end{equation*}
Hence, for almost all $t\in (0,T_0)$ and for all $\varepsilon>0$, we
have
\begin{equation}\label{w181}
\begin{split}
&\frac{1}{2}\int_\O(\tilde{\r}|\u-\tilde{\u}|^2+|\nabla\d-\nabla\tilde{\d}|^2)\ d\x+\int_0^t\int_\O(|\nabla\u-\nabla\tilde{\u}|^2+|\triangle\d-\triangle\tilde{\d}|^2)\ d\x d\tau\\
&\leq
\int_0^t\left(C_\varepsilon\|\partial_\tau\u+\u\cdot\nabla\u\|_{L^3}^2\|\r-\tilde{\r}\|_{L^2}^2
+\varepsilon\|\u-\tilde{\u}\|_{L^6}^2\right)\
d\tau+\varepsilon\int_0^t\int_\O
|\triangle\d-\triangle\tilde{\d}|^2\ d\x
d\tau\\
&\quad+C_\varepsilon\int_0^t\big(\|\nabla\d\|_{L^\infty}^4\int_\O|\d-\tilde{\d}|^2\
d\x\big)\
d\tau+\int_0^t\big(\|\nabla\u\|_{L^\infty}\int_\O\tilde{\r}|\u-\tilde{\u}|^2\
d\x\big)\
d\tau\\
&\quad+\int_0^t\left(C_\varepsilon(\|\nabla\d+\nabla\tilde{\d}\|_{L^\infty}^2+\|\u\|_{L^\infty}^2+\|\triangle\d\|^2_{L^3})\int_\O|\nabla\d-\nabla\tilde{\d}|^2\
d\x\right)d\tau.
\end{split}
\end{equation}
Here we have used H\"{o}lder's inequality and Cauchy's inequality with $\varepsilon$.

Now, we wish to estimate $\|\r-\tilde{\r}\|_{L^2}$ and
$\|\d-\tilde{\d}\|_{L^2}$ . We write
\begin{equation}\label{w191}
\partial_t(\r-\tilde{\r})+\nabla(\r-\tilde{\r})\cdot\tilde{\u}=(\tilde{\u}-\u)\cdot\nabla\r,
\end{equation}
and
\begin{equation}\label{w1911}
\partial_t(\d-\tilde{\d})+\u\cdot\nabla(\d-\tilde{\d})+(\u-\tilde{\u})\cdot\nabla\tilde{\d}=\triangle\d-\triangle\tilde{\d}+|\nabla\d|^2\d-|\nabla\tilde{\d}|^2\tilde{\d}.
\end{equation}
Multiply \eqref{w191}\big(resp. \eqref{w1911}\big) by
$\r-\tilde{\r}$ \big(resp. $\d-\tilde{\d}$\big) and integrate over
$\O\times (0,t)$, we have
\begin{equation*}
\frac{1}{2}\int_\O|\r-\tilde{\r}|^2\
d\x=\int_0^t\int_\O(\r-\tilde{\r})(\tilde{\u}-\u)\cdot\nabla\r\ d\x
d\tau,
\end{equation*}
and
\begin{equation*}
\begin{split}
&\frac{1}{2}\int_\O|\d-\tilde{\d}|^2\ d\x\\
&=-\int_0^t\int_{\O}(\u-\tilde{\u})\cdot\nabla\tilde{\d}\cdot(\d-\tilde{\d})\
d\x d\tau-\int_0^t\int_{\O}|\nabla\d-\nabla\tilde{\d}|^2\ d\x
d\tau\\
&\quad+\int_0^t\int_{\O}|\nabla\d|^2|\d-\tilde{\d}|^2\ d\x d\tau
+\int_0^t\int_{\O}(\nabla\d+\nabla\tilde{\d}):(\nabla\d-\nabla\tilde{\d})\
\tilde{\d}\cdot(\d-\tilde{\d})\ d\x d\tau.
\end{split}
\end{equation*}
 Employing the same argument as \eqref{w181}, we get
\begin{equation}\label{w1921}
\begin{split}
&\frac{1}{2}\int_\O|\r-\tilde{\r}|^2\ d\x\leq
\int_0^t(C_\varepsilon\|\nabla\r\|_{L^3}^2\|\r-\tilde{\r}\|_{L^2}^2+\varepsilon
\|\u-\tilde{\u}\|_{L^6}^2)\ d\tau,
\end{split}
\end{equation}
\begin{equation}\label{w1922}
\begin{split}
&\frac{1}{2}\int_\O|\d-\tilde{\d}|^2\ d\x\\
&\leq
\int_0^t\left(C_\varepsilon\|\nabla\tilde{\d}\|_{L^3}^2+\|\nabla\d\|_{L^\infty}^2+\frac{1}{2}\right)\int_{\O}|\d-\tilde{\d}|^2\
d\x d\tau+\varepsilon\int_0^t\|\u-\tilde{\u}\|_{L^6}^2\
d\tau\\
&\quad+\int_0^t\left(1+\frac{\|\nabla\d+\nabla\tilde{\d}\|^2_{L^\infty}}{2}\right)\int_{\O}|\nabla\d-\nabla\tilde{\d}|^2\
d\x d\tau.
\end{split}
\end{equation}
 Using Sobolev's inequality $\|\u-\tilde{\u}\|_{L^6}\leq
C\|\nabla\u-\nabla\tilde{\u}\|_{L^2}$, we eventually get, from
\eqref{w181}, \eqref{w1921} and \eqref{w1922}, for almost all $t\in
(0,T_0)$,
\begin{equation*}
\begin{split}
&\int_\O(\tilde{\r}|\u-\tilde{\u}|^2+|\nabla\d-\nabla\tilde{\d}|^2+|\r-\tilde{\r}|^2+|\d-\tilde{\d}|^2)\ d\x\\
&+\int_0^t\int_\O(|\nabla\u-\nabla\tilde{\u}|^2+|\triangle\d-\triangle\tilde{\d}|^2)\ d\x d\tau\\
&\leq\int_0^t\int_\O\Big(C_\varepsilon(\tau)\big(|\r-\tilde{\r}|^2+|\nabla\d-\nabla\tilde{\d}|^2+|\d-\tilde{\d}|^2\big)\\
&\qquad\qquad\quad+C(\tau)\big(\tilde{\r}|\u-\tilde{\u}|^2+|\nabla\d-\nabla\tilde{\d}|^2+|\d-\tilde{\d}|^2\big)\Big)\
d\x d\tau,
\end{split}
\end{equation*}
where $C_\varepsilon(\cdot),C(\cdot)$ denote various non-negative
measurable functions in $L^1(0,T_0)$ occurred  when we   applied the parabolic type estimates
for quasi-linear equations  (cf. \cite{OAL},
Chapter VI, Section 2) to \eqref{e23} to obtain  $\tilde{\d}(\cdot,
t)\in C^{1, \alpha}$  for some
$\alpha>0$, with the $C^{1, \alpha}$ norm   independent of $t$. We
hence conclude that $\u=\tilde{\u},\ \d=\tilde{\d}$ and
$\r=\tilde{\r}$ a.e. in $\O\times (0,T_0)$, by applying
Gr\"{o}nwall's inequality.

The proof of Theorem \ref{T2} is complete.

\bigskip\bigskip

\section*{Acknowledgments}
D. Wang's research was supported in part by the National Science
Foundation under Grant  DMS-0906160, and by the
Office of Naval Research under Grant N00014-07-1-0668.

\end{document}